\documentclass[11pt]{article} 
\usepackage{color}
\usepackage[leqno]{amsmath}
\usepackage{amssymb,amsthm,url}
\usepackage{shuffle}
\usepackage{mathrsfs}
\usepackage{tikz}
\usetikzlibrary{arrows}
\usepackage[letterpaper,margin=1in]{geometry}
\usepackage{enumitem}		

\newcommand\T{\mathbf{T}}
\newcommand\U{\mathbf{U}}
\newcommand\D{\mathbf{D}}
\newcommand\f{\mathbf{f}}
\newcommand\g{\mathbf{g}}
\DeclareMathOperator{\id}{id}

\DeclareMathOperator{\Ind}{Ind}
\DeclareMathOperator{\Res}{Res}
\DeclareMathOperator{\sh}{shape}
\DeclareMathOperator{\std}{std}
\DeclareMathOperator{\RSK}{RSK}

\DeclareMathOperator{\Des}{Des}
\DeclareMathOperator{\ides}{ides}

\newcommand\calsh{\mathcal{S}}
\newcommand\calh{\mathcal{H}}
\newcommand\calhn{\mathcal{H}_n}

\newcommand\barcalh{\bar{\mathcal{H}}}

\newcommand\calb{\mathcal{B}}
\newcommand\calbn{\mathcal{B}_n}
\newcommand\hatcalb{\check{\mathcal{B}}}

\newcommand\calbdual{\mathcal{B}^*}

\newcommand\barcalb{\bar{\mathcal{B}}}
\newcommand\barcalbn{\bar{\mathcal{B}}_n}

\newcommand\hatk{\check{K}}
\newcommand\barg{\bar G}
\newcommand\barv{\bar V}
\newcommand\bark{\bar K}
\newcommand\hatbark{\check{\bar{K}}}
\newcommand\barx{\bar x}
\newcommand\bary{\bar y}

\newcommand\fsym{\mathbf{FSym}}
\newcommand\fqsym{\mathbf{FQSym}}
\newcommand\sn{\mathfrak{S}_n}
\newcommand\sd{\mathfrak{S}_d}
\newcommand\snminusone{\mathfrak{S}_{n-1}}
\newcommand\snminusk{\mathfrak{S}_{n-k}}
\newcommand\si{\mathfrak{S}_i}
\newcommand\sj{\mathfrak{S}_j}

\date{}

\makeatletter
\theoremstyle{plain}
\newtheorem{thm}{\protect\theoremname}[section]
  \theoremstyle{definition}
  \newtheorem{defn}[thm]{\protect\definitionname}
 \newcommand\thmsname{\protect\theoremname}
 \newcommand\nm@thmtype{theorem}
 \theoremstyle{plain}

  \theoremstyle{plain}
  
  \theoremstyle{plain}
  
  \theoremstyle{plain}
  
  \theoremstyle{remark}
  \newtheorem*{rem*}{\protect\remarkname}
  \theoremstyle{definition}
  \newtheorem{example}[thm]{\protect\examplename}
 \theoremstyle{remark}
 \newtheorem*{rems*}{\protect\remarksname}
{\settowidth{\lyxlabelwidth}{#2}
\begin{description}[font=\normalfont,style=sameline,
leftmargin=\lyxlabelwidth,#1]}
{\end{description}}
  \providecommand{\remarksname}{Remarks}

\newlength\cellsize 
\savebox2{%
\begin{picture}(12,12)
\put(0,0){\line(1,0){12}}
\put(0,0){\line(0,1){12}}
\put(12,0){\line(0,1){12}}
\put(0,12){\line(1,0){12}}
\end{picture}}

\newcommand\cellify[1]{\def\thearg{#1}\def\nothing{}%
\ifx\thearg\nothing
\vrule width0pt height\cellsize depth0pt\else
\hbox to 0pt{\usebox2\hss}\fi%
\vbox to 12\unitlength{
\vss
\hbox to 12\unitlength{\hss$#1$\hss}
\vss}}

\newcommand\tableau[1]{\vtop{\let\\=\cr
\setlength\baselineskip{-12000pt}
\setlength\lineskiplimit{12000pt}
\setlength\lineskip{0pt}
\halign{&\cellify{##}\cr#1\crcr}}}

\newcommand{\e}{\mbox{}}

\makeatother

  \providecommand{\definitionname}{Definition}
  \providecommand{\examplename}{Example}
  \providecommand{\remarkname}{Remark}
\providecommand{\theoremname}{Theorem}

\begin{document}

\title{Lumpings of Algebraic Markov Chains arise from Subquotients}

\author{C. Y. Amy Pang}

\maketitle

\begin{abstract}
A function on the state space of a Markov chain is a ``lumping''
if observing only the function values gives a Markov chain. We give
very general conditions for lumpings of a large class of algebraically-defined
Markov chains, which include random walks on groups and other common
constructions. We specialise these criteria to the case of descent
operator chains from combinatorial Hopf algebras, and, as an example,
construct a ``top-to-random-with-standardisation'' chain on permutations
that lumps to a popular restriction-then-induction chain on partitions,
using the fact that the algebra of symmetric functions is a subquotient
of the Malvenuto-Reutenauer algebra.
\end{abstract}

\section{Introduction\label{sec:Introduction}}

Combinatorialists have built a variety of frameworks for studying
Markov chains algebraically, most notably the theories of random walks
on groups \cite{randomwalksongroups,randomwalksongroupspersibook}
and their extensions to monoids \cite{hyperplanewalk,lrb,rtrivialmonoids}.
Further examples include \cite{jasondownup,iwahoriheckealgmetrowalk,descentoperatorchains}.
These approaches usually associate some algebraic operator to the
Markov chain, the advantage being that the eigendata of the operator
reflects the convergence rates of the chains. As an easy example,
consider $n$ playing cards laid in a row on a table, and imagine
exchanging two randomly chosen cards at each time step. \cite{randomtranspositions}
represents this random transposition shuffle as multiplication by
the sum of transpositions in the group algebra of the symmetric group,
and deduces from the representation theory of the symmetric group
that asymptotically $\frac{1}{2}n\log n$ steps are required to randomise
the order of the cards. See Example \ref{ex:cardshuffleIA} below
for more details of the setup.

Analogous to how these frameworks translated the convergence rate
calculation into an algebraic question of representations and characters,
the present paper gives very general algebraic conditions for a different
probability problem: when is a function $\theta$ of an algebraic
Markov chain $\{X_{t}\}$ a Markov function, meaning that the sequence
of random variables $\{\theta(X_{t})\}$ is itself a Markov chain?
The motivation for this is that often, only certain functions of Markov
chains are of interest. For example, the random transposition shuffle
above may be in preparation for a card game that only uses the cards
on the left (Example \ref{ex:cardshuffleIB_topcards}), or where only
the position of one specific card is important (Example \ref{ex:cardshuffleIB_aceofspades}).
One naturally suspects that randomising only half the cards or only
the position of one card may take fewer than $\frac{1}{2}n\log n$
moves, as these functions can become randomised before the full chain
does. The convergence rates of such functions of Markov chains are
generally easier to analyse when the function is Markov, as the \emph{lumped
chain} $\{\theta(X_{t})\}$ can be studied independently of the full
chain $\{X_{t}\}$. The reverse problem is also interesting: a Markov
chain $X'_{t}$ that is hard to analyse directly may benefit from
being viewed as $\{\theta(X_{t})\}$ for a more tractable ``lift''
chain $\{X_{t}\}$. \cite{paseplift,paseplift2,bernoullilaplace}
are examples of this idea.

The aim of this paper is to expedite the search for lumpings and lifts
of ``algebraic'' Markov chains by giving very general conditions
for their existence. As formalised in Section I.A, the chains under
consideration are associated to a linear transformation $T:V\rightarrow V$,
where the state space is a basis $\calb$ of the vector space $V$.
Our two main discoveries for such chains are:
\begin{itemize}
\item (Section \ref{sec:tableaux}/I.B, Theorem \ref{thm:stronglumping-linearmap})
if $T$ descends to a well-defined map $\bar{T}$ on a quotient space
$\bar{V}$ of $V$ that ``respects the basis'' $\calb$, then the
quotient projection $\theta:V\rightarrow\bar{V}$ gives a lumping
from any initial distribution on $\calb$. The lumped chain is associated
to $\bar{T}:\barv\rightarrow\barv$.
\item (Section \ref{sec:permutations}/I.C, Theorem \ref{thm:weaklumping-linearmap})
if $V$ contains a $T$-invariant subspace $V'$ that ``respects
the basis'' $\calb$, then $T:V'\rightarrow V'$ corresponds to a
lumping that is only valid for certain initial distributions, i.e.
a \emph{weak lumping}.
\end{itemize}
Part I/Section 2%
\footnote{(The sections have both custom numbering and standard numerical numbering,
to be consistent with the journal version.)%
} states and proves the above very general theorems, and illustrates
them with numerous simple examples, both classical and new.

Part II/Section 3 specialises these general lumping criteria to descent
operator chains on combinatorial Hopf algebras \cite{descentoperatorchains}
- in essence, lumpings from any initial distribution correspond to
quotient algebras, and weak lumpings to subalgebras. This is applied
to two fairly elaborate examples. Sections II.A-II.D demystifies a
theorem of Jason Fulman \cite[Th. 3.1]{jasonlift}: the probability
distribution of the RSK shape \cite[Sec. 7.11]{stanleyec2}\cite[Sec. 4]{fultonyoungtableaux}
of a permutation, after $t$ top-to-random shuffles (Example \ref{ex:cardshuffleIA}
below) from the identity, agrees with the probability distribution
of a partition after $t$ steps of a certain Markov chain that removes
then readds a random box (see Section \ref{sub:partitions}/the second
half of Section II.A). Fulman remarked that the connection between
these two chains was ``surprising'', perhaps because it is not a
lumping (see the start of Section 3/Part II). Here we use the new
lumping criteria for descent operator chains to construct a similar
chain to top-to-random shuffling that does lump to the chain on partitions,
and prove that its probability distribution after $t$ steps from
the identity agrees with that of top-to-random.

The second application, in Section \ref{sub:descentsetlump}/II.E,
is a Hopf-algebraic re-proof of a result of Christos Athanasiadis
and Persi Diaconis \cite[Ex. 5.8]{hyperplanelump}, that riffle-shuffles
and related card shuffles lump via descent set.

\subsection*{Acknowledgements}

I would like to thank Nathan Williams for a question that motivated
this research, and Persi Diaconis, Jason Fulman and Franco Saliola
for numerous helpful conversations, and Federico Ardila, Gr\'{e}gory
Ch\^{a}tel, Mathieu Guay-Paquet, Simon Rubenstein-Salzedo, Yannic
Vargas and Graham White for useful comments. SAGE computer software
\cite{sage} was very useful, especially the combinatorial Hopf algebras
coded by Aaron Lauve and Franco Saliola.

\section{Part I: General Theory}

\subsection{Matrix notation}

Given a matrix $A$ , let $A(x,y)$ denote its entry in row $x$,
column $y$, and write $A^{T}$ for the transpose of $A$.

Let $V$ be a vector space (over $\mathbb{R}$) with basis $\calb$,
and $\T:V\rightarrow V$ be a linear map.  Write $\left[\T\right]_{\calb}$
for the matrix of $\T$ with respect to $\calb$ . In other words,
the entries of $\left[\T\right]_{\calb}$ satisfy
\[
\T(x)=\sum_{y\in\calb}\left[\T\right]_{\calb}(y,x)y
\]
for each $x\in\calb$.

$V^{*}$ is the \emph{dual vector space} to $V$, the set of linear
functions from $V$ to $\mathbb{R}$. Its natural basis is $\calbdual:=\left\{ x^{*}|x\in\calb\right\} $,
where $x^{*}$ satisfies $x^{*}(x)=1$, $x^{*}(y)=0$ for all $y\in\calb$,
$y\neq x$. The \emph{dual map} to $\T:V\rightarrow V$ is the linear
map $\T^{*}:V^{*}\rightarrow V^{*}$ satisfying $(\T^{*}f)(v)=f(\T v)$
for all $v\in V,f\in V^{*}$. Dualising a linear map is equivalent
to transposing its matrix: $\left[\T^{*}\right]_{\calbdual}=\left[\T\right]_{\calb}^{T}$.

\subsection{I.A: Markov Chains from Linear Maps via the Doob Transform\label{sec:partitions}}

To start, here is a quick summary of the Markov chain facts required
for this work. A (discrete time) Markov chain is a sequence of random
variables $\{X_{t}\}$, where each $X_{t}$ belongs to the \emph{state
space} $\Omega$. All Markov chains here are time-independent and
have a finite state space. Hence they are each described by an $|\Omega|$-by-$|\Omega|$
\emph{transition matrix} $K$: for any time $t$,
\[
P\{X_{t}=x_{t}|X_{0}=x_{0},X_{1}=x_{1},\dots,X_{t-1}=x_{t-1}\}=P\{X_{t}=x_{t}|X_{t-1}=x_{t-1}\}:=K(x_{t-1},x_{t}).
\]
(Here, $P\{X|Y\}$ is the probability of event $X$ given event $Y$.)
If the probability distribution of $X_{t}$ is expressed as a row
vector $g_{t}$, then taking one step of the chain is equivalent to
multiplication by $K$ on the right: $g_{t}=g_{t-1}K$. (Some authors,
notably \cite{rtrivialmonoids}, take the opposite convention, where
$P\{X_{t}=y|X_{t-1}=x\}:=K(y,x)$, and the distribution of $X_{t}$
is represented by a column vector $f_{t}$ with $f_{t}=Kf_{t-1}$.)
Note that a matrix $K$ specifies a Markov chain in this manner if
and only if $K(x,y)\geq0$ for all $x,y\in\Omega$, and $\sum_{y\in\Omega}K(x,y)=1$
for each $x\in\Omega$. A probability distribution $\pi:\Omega\rightarrow\mathbb{R}$
is a\emph{ stationary distribution} if it satisfies $\sum_{x\in\Omega}\pi(x)K(x,y)=\pi(y)$
for each state $y\in\Omega$. We refer the reader to the textbooks
\cite{markovmixing,lumping} for more background in Markov chain theory.

This paper concerns chains which arise from linear maps. A simple
motivating example is a random walk on a group. Given a probability
distribution $Q$ on a group $G$ (i.e. a function $Q:G\rightarrow\mathbb{R}$),
consider the following Markov chain on the state space $\Omega=G$:
at each time step, choose a group element $g$ with probability $Q(g)$,
and move from the current state $x$ to the state $xg$. This chain
is associated to the ``right multiplication by $\sum_{g\in G}Q(g)g$''
operator on the group algebra $\mathbb{R}G$, i.e. to the linear transformation
$\T:\mathbb{R}G\rightarrow\mathbb{R}G$, $\T(x):=x\left(\sum_{g\in G}Q(g)g\right)$.
To state the relationship more precisely: the transition matrix of
the Markov chain is the transpose of the matrix of $\T$ relative
to the basis $G$ of $\mathbb{R}G$, i.e. $K=[\T]_{G}^{T}$. (One
can define similar chains from left-multiplication operators.)

Before generalising this connection to other linear operators, here
are some simple examples of random walks on groups that we will use
to illustrate lumpings in later sections.
\begin{example}[Card-shuffling]
 \label{ex:cardshuffleIA} Random walks on $G=\sn$, the symmetric
group, describe many examples of card-shuffling. The state space of
these chains are the $n!$ possible orderings of a deck of $n$ cards.
For convenience, suppose the cards are labelled $1,2,\dots,n$, each
label occurring once. There are various different conventions on how
to represent such an ordering by a permutation, see \cite[Sec. 2.2]{yufeizhao}.
We follow the more modern notation in \cite{riffleshufflerepeatedcards}
(as opposed to \cite{toptorandom,originalriffleshuffle}) and associate
$\sigma$ to the ordering where $\sigma(1)$ is the label of the top
card, $\sigma(2)$ is the label of the second card from the top, ...,
$\sigma(n)$ is the label of the bottom card. In other words, writing
$\sigma$ in one-line notation $(\sigma(1),\sigma(2),\dots,\sigma(n))$
(see Section \ref{sec:combinotation}/the start of Part II) lists
the card labels from top to bottom. Observe that in this convention,
right-multiplication by a permutation $\tau$ moves the card in position
$\tau(i)$ to position $i$.

Two simple shuffles that we will consider are:
\begin{itemize}
\item \emph{top-to-random} \cite{toptorandom,cppriffleshuffle}: remove
the top card, then reinsert it into the deck at one of the $n$ possible
positions, chosen uniformly. A possible trajectory for a deck of five
cards is

\noindent 
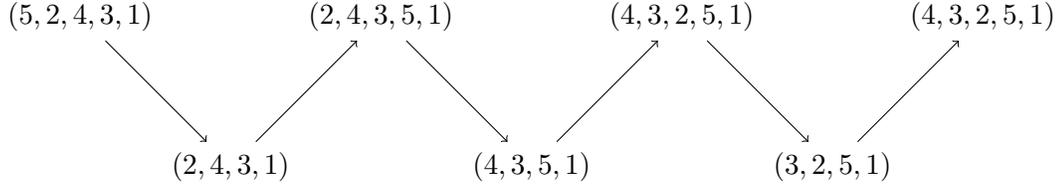
\begin{center}
\begin{tikzpicture} 
\node (A) at (0,0) {$(5,2,4,3,1)$}; 
\node (B) at (4,0)  {$(2,4,3,5,1)$}; 
\node (C) at (8,0)  {$(4,3,2,5,1)$};  
\node (D) at (12,0) {$(4,3,2,5,1)$}; 
\node (E) at (2,-2) {$(2,4,3,1)$}; 
\node (F) at (6,-2)  {$(4,3,5,1)$}; 
\node (G) at (10,-2)  {$(3,2,5,1)$}; 
\draw[->] (A) -- (E); 
\draw[->] (E) -- (B); 
\draw[->] (B) -- (F); 
\draw[->] (F) -- (C); 
\draw[->] (C) -- (G); 
\draw[->] (G) -- (D); 
\end{tikzpicture}
\par\end{center}

The associated distribution $Q$ on $\sn$ is 
\[
Q(g)=\begin{cases}
\frac{1}{n} & \mbox{ if }g=(i\ i-1\ \dots\ 1)\mbox{ for some }i\mbox{, in cycle notation};\\
0 & \mbox{otherwise.}
\end{cases}
\]
(Note that the identity permutation is the case $i=1$.) Equivalently,
the associated linear map is right-multiplication by 
\[
q=\frac{1}{n}\sum_{i=1}^{n}(i\ i-1\ \dots\ 1).
\]
This Markov chain has been thoroughly analysed over the literature:
\cite[Sec. 1, Sec. 2]{stronguniformtime} uses a strong uniform time
to elegantly show that roughly $n\log n$ iterations are required
to randomise the deck, and \cite[Cor. 2.1]{cppriffleshuffle} finds
the explicit probabilities of achieving a particular permutation after
any given number of shuffles. The time-reversal of the top-to-random
shuffle is the equally well-studied Tsetlin library \cite{tsetlin}:
\cite[Sec. 4.6]{oneminuse} describes an explicit algorithm for an
eigenbasis, \cite{phatarfod} derives the spectrum for a weighted
version, and \cite{tsetlinlibrary} lists many more references. 

\item \emph{random-transposition} \cite[Chap. 3D]{randomtranspositions,randomwalksongroupspersibook}:
choose two cards, possibly with repetition, uniformly and independently.
If the same card was chosen twice, do nothing. Otherwise, exchange
the two chosen cards. A possible trajectory for a deck of five cards
is

\noindent \begin{center}
\begin{tikzpicture} 
\node (A) at (0,0) {$(5,2,4,3,1)$}; 
\node (B) at (4,0)  {$(5,4,2,3,1)$}; 
\node (C) at (8,0)  {$(5,4,1,3,2)$};  
\node (D) at (12,0) {$(4,5,1,3,2)$}; 
\draw[->] (A) -- (B); 
\draw[->] (B) -- (C); 
\draw[->] (C) -- (D); 
\end{tikzpicture}
\par\end{center}

The associated distribution $Q$ on $\sn$ is 
\[
Q(g)=\begin{cases}
\frac{1}{n} & \mbox{ if }g\mbox{ is the identity};\\
\frac{2}{n^{2}} & \mbox{ if }g\mbox{ is a transposition;}\\
0 & \mbox{otherwise.}
\end{cases}
\]
Equivalently, the associated linear map is right-multiplication by
\[
q=\frac{1}{n}\id+\frac{2}{n}\sum\sigma;
\]
summing over all transpositions $\sigma$.

The mixing time for the random-transposition shuffle is $\frac{1}{2}n\log n$,
as shown in \cite{randomtranspositions} using the representation
theory of the symmetric group. \cite{randomwalkonmatchings} uses
this chain to induce Markov chains on trees and on matchings. A recent
extension to random-involutions is \cite{randominvolutions}.

\end{itemize}
\end{example}

\begin{example}[Flip a random bit]
 \label{ex:dbitbinaryIA} \cite{polytopewalk}: Let $G=\left(\mathbb{Z}/2\mathbb{Z}\right)^{d}$,
written additively as binary strings of length $d$. At each time
step, uniformly choose one of the $d$ bits, and change it either
from 0 to 1 or from 1 to 0. A possible trajectory for $d=5$ is

\noindent \begin{center}
\begin{tikzpicture} 
\node (A) at (0,0) {$(1,0,0,1,1)$}; 
\node (B) at (4,0)  {$(1,0,1,1,1)$}; 
\node (C) at (8,0)  {$(1,0,1,1,0)$};  
\node (D) at (12,0) {$(0,0,1,1,0)$}; 
\draw[->] (A) -- (B); 
\draw[->] (B) -- (C); 
\draw[->] (C) -- (D); 
\end{tikzpicture}
\par\end{center}

The associated distribution $Q$ is 
\[
Q(g)=\begin{cases}
\frac{1}{d} & \mbox{ if }g\mbox{ consists of }d-1\mbox{ zeroes and }1\mbox{ one};\\
0 & \mbox{otherwise.}
\end{cases}
\]
Equivalently, the associated linear map is right-multiplication by
\[
q=\frac{1}{d}\left((1,0,\dots,0)+(0,1,0,\dots,0)+\dots+(0,\dots,0,1)\right).
\]
(The addition in $q$ is in the group algebra $\mathbb{R}G$, not
within the group $G$.)

\cite{polytopewalk} investigated the return probabilities of this
walk and similar walks on $\left(\mathbb{Z}/2\mathbb{Z}\right)^{d}$
that allow changing more than one bit.

\end{example}
As detailed above, the transition matrix of a random walk on a group
is $K=[\T]_{G}^{T}$, where $\T$ is the right-multiplication operator
on the group algebra $\mathbb{R}G$ defined by $\T(x):=x\left(\sum_{g\in G}Q(g)g\right)$.
This relationship between $K$ and $\T$ allows the representation
theory of $G$ to illuminate the converge rates of the chain. A naive
generalisation is to declare new transition matrices to be $K:=[\T]_{\calb}^{T}$,
for other linear transformations $\T$ on a vector space with basis
$\calb$. The state space of such a chain is the basis $\calb$, and
the intuition is that the transition probabilities $K(x,y)$ would
represent the chance of obtaining $y$ when applying $\T$ to $x$.

In order for $K:=[\T]_{\calb}^{T}$ to be a transition matrix, we
require $K(x,y)\geq0$ for all $x,y$ in $\calb$, and $\sum_{y\in\calb}K(x,y)=1$.
As noted by Persi Diaconis (personal communication), the non-negativity
condition can be achieved by adding multiples of the identity transformation
to $\T$, which essentially keeps the eigendata properties in Theorem
\ref{thm:doob-transform} below. In any case, many linear operators
arising from combinatorics already have non-negative coefficients
with respect to natural bases, so we do not dwell on this problem.

The row-sum condition $\sum_{y\in\calb}K(x,y)=1$, though true for
many important cases \cite{hyperplanewalk,rtrivialmonoids}, is less
guaranteed. There are many possible ways to adjust $K$ so that its
row sums become 1. One way which preserves the relationship between
the eigendata of $\T$ and the convergence rates of the chain is to
rescale $\T$ and the basis $\calb$ using the Doob $h$-transform.

The Doob $h$-transform is a very general tool in probability, used
to condition a process on some event in the future \cite{doobtransformoriginal}.
The simple case of relevance here is conditioning a (finite, discrete-time)
Markov chain on non-absorption. The Doob transform constructs the
transition matrix of the conditioned chain out of the transition probabilities
of the original chain between non-absorbing states, or, equivalently,
out of the original transition matrix with the rows and columns for
absorbing states removed. As observed in the multiple references below,
the same recipe essentially works for any arbitrary non-negative matrix
$K$. 

The Doob transform relies on a positive \emph{right eigenfunction}
$\eta$ of $K$, i.e. a positive function $\eta:\calb\rightarrow\mathbb{R}$
satisfying $\sum_{y}K(x,y)\eta(y)=\beta\eta(x)$ for some positive
number $\beta$, which is the eigenvalue. (Functions satisfying this
condition with $\beta=1$ are called \emph{harmonic}, hence the name
$h$-transform.) To say this in a basis-independent way, recall that
$K=[\T]_{\calb}^{T}=\left[\T^{*}\right]_{\calbdual}$, so $\eta$
(or more accurately, its linear extension in $V^{*}$) is an eigenvector
of the dual map $\T^{*}:V^{*}\rightarrow V^{*}$ with eigenvalue $\beta$,
i.e. $\eta\circ T=\beta\eta$ as functions on $V$.
\begin{thm}[Markov chains from linear maps via the Doob $h$-transform]
 \label{thm:doob-transform} \cite[Def. 8.11, 8.12]{doobtransformbook}\cite[Sec. 17.6.1]{markovmixing}\cite[Lem. 4.4.1]{zhou}\cite[Lem. 1.4, Lem. 2.11]{doobnotes}
Let $V$ be a finite-dimensional vector space with basis $\calb$,
and $\T:V\rightarrow V$ be a linear map for which $K:=\left[\T\right]_{\calb}^{T}$
has all entries non-negative. Suppose $K$ has a positive right eigenfunction
$\eta$ with eigenvalue $\beta>0$. Then
\begin{enumerate}[label=\roman*.]
\item The matrix
\[
\hatk(x,y):=\frac{1}{\beta}K(x,y)\frac{\eta(y)}{\eta(x)}
\]
is a transition matrix. Equivalently, $\hatk:=\left[\frac{\T}{\beta}\right]_{\hatcalb}^{T}$,
where $\hatcalb:=\left\{ \frac{x}{\eta(x)}:x\in\calb\right\} $.
\item The left eigenfunctions $\g$ for $\hatk$, with eigenvalue $\alpha$
(i.e. $\sum_{x}\g(x)K(x,y)=\alpha\g(y)$), are in bijection with the
eigenvectors $g\in V$ of $\T$, with eigenvalue $\frac{\alpha}{\beta}$,
via 
\[
\g(x):=\eta(x)\times\mbox{coefficient of }x\mbox{ in }g.
\]

\item The stationary distributions $\pi$ for $\hatk$ are precisely the
functions of the form 
\[
\pi(x):=\eta(x)\frac{\xi_{x}}{\sum_{x\in\calb}\xi_{x}\eta(x)},
\]
where $\sum_{x\in\calb}\xi_{x}x\in V$ is an eigenvector of $\T$
with eigenvalue 1, whose coefficients $\xi_{x}$ are all non-negative.
\item The right eigenfunctions $\f$ for $\hatk$, with eigenvalue $\alpha$
(i.e. $\sum_{y}K(x,y)\f(y)=\alpha\f(x)$), are in bijection with the
eigenvectors $f\in V^{*}$of the dual map $\T^{*}$, with eigenvalue
$\frac{\alpha}{\beta}$, via 
\[
\f(x):=\frac{1}{\eta(x)}f(x).
\]

\end{enumerate}
\end{thm}
The function $\eta:V\rightarrow\mathbb{R}$ above is called the \emph{rescaling
function}. The output $\hatk$ of the transform depends on the choice
of rescaling function. Observe that, if $K:=\left[\T\right]_{\calb}^{T}$
already has each row summing to 1, then the constant function $\eta\equiv1$
on $\calb$ is a positive right eigenfunction of $K$ with eigenvalue
1, and using this constant rescaling function results in no rescaling
at all: $\hatk=K$. This will be the case in all examples in Sections
\ref{sec:tableaux}/I.B and \ref{sec:permutations}/I.C, so the reader
may wish to skip the remainder of this section on first reading, and
assume $\eta\equiv1$ on $\calb$ in all theorems. (Note that $\eta\equiv1$
on $\calb$ does not mean $\eta$ is constant on $V$, since $\eta$
is a linear function. Instead, $\eta$ sends a vector in $V$ to the
sum of the coefficients when $v$ is expanded in the $\calb$ basis.) 
\begin{proof}
To prove i, first note that $\hatk(x,y)\ge0$ because $\beta>0$ and
$\eta(x)>0$ for all $x$. And the rows of $\hatk$ sum to 1 because
\[
\sum_{y}\hatk(x,y)=\frac{\sum_{y}K(x,y)\eta(y)}{\beta\eta(x)}=\frac{\beta\eta(x)}{\beta\eta(x)}=1.
\]

Parts ii and iv are immediate from the definition of $\hatk$. To
see part iii, recall that a stationary distribution is precisely a
positive left eigenfunction with eigenvalue 1 (normalised to be a
distribution).\end{proof}
\begin{example}
To illustrate the Doob transform, here is the down-up chain on partitions
of size 3. (See Section \ref{sub:partitions}/II.A for a general description,
and an interpretation in terms of restriction and induction of representations
of symmetric groups.)

Let $V_{3}$ be the vector space with basis $\calb_{3}$, consisting
of the three partitions of size 3:\begin{equation*} 
\calb_3 :=\left\{ \raisebox{1.8ex}{\tableau{ \e & \e & \e}, \, \tableau{ \e & \e \\ \e}, \, \tableau{ \e \\ \e \\ \e}} \right\}.
\end{equation*} We define below a second vector space $V_{2}$, and linear transformations
$\D:V_{3}\rightarrow V_{2}$, $\U:V_{2}\rightarrow V_{3}$ whose composition
$\T=\U\circ\D$ will define our Markov chain on $\calb_{3}$.

$V_{2}$ is the vector space with basis $\calb_{2}$, the two partitions
of size 2\begin{equation*} 
\calb_2 :=\left\{ \raisebox{0.6ex}{\tableau{ \e & \e}, \, \tableau{ \e \\ \e} } \right\}.
\end{equation*}  

For $x\in\calb_{3}$, define $\D(x)$ to be the sum of all elements
of $\calb_{2}$ which can be obtained from $x$ by deleting a box
on the right end of any row. So\begin{align*} 
\D \left( \raisebox{-0.6ex}{\tableau{ \e & \e & \e }}\right) &= \raisebox{-0.6ex}{ \tableau{ \e & \e} };\\
\D \left( \raisebox{0.6ex}{\tableau{ \e & \e \\ \e}}\right) &= \raisebox{0.6ex}{ \tableau{ \e & \e} + \tableau{ \e \\ \e}  };\\
\D \left( \raisebox{1.8ex}{\tableau{ \e \\ \e \\ \e}}\right) &= \raisebox{1.8ex}{\tableau{ \e \\ \e}  };
\end{align*}Then, for $x\in\calb_{2}$, define $\U(x)$ to be the sum of all elements
of $\calb_{3}$ which can be obtained from $x$ by adding a new box
on the right end of any row, including on the row below the last row
of $x$. So\begin{align*} 
\U \left( \raisebox{-0.6ex}{\tableau{ \e & \e }}\right) &= \raisebox{-0.6ex}{\tableau{ \e & \e & \e } + \tableau{ \e & \e \\ \e} };\\
\U \left( \raisebox{0.6ex}{\tableau{ \e \\ \e }}\right) &= \raisebox{1.8ex}{\tableau{ \e & \e \\ \e} + \tableau{ \e \\ \e \\ \e} };
\end{align*}

An easy calculation shows that 
\[
K=\left[\U\circ\D\right]_{\calb_{3}}^{T}=\begin{bmatrix}1 & 1 & 0\\
1 & 2 & 1\\
0 & 1 & 1
\end{bmatrix},
\]
which has all entries non-negative, but its rows do not sum to 1.

The function $\eta:\calb_{3}\rightarrow\mathbb{R}$ defined by  \begin{equation*}  \eta\left( \raisebox{-0.6ex}{\tableau{ \e & \e & \e}}\right)=1; \,  \eta\left( \raisebox{0.6ex}{\tableau{ \e & \e \\ \e}}\right)=2; \,  \eta \left( \raisebox{1.8ex}{\tableau{ \e \\ \e \\ \e}}\right)=1 \end{equation*} is
a right eigenfunction of $K$ with eigenvalue $\beta=3$. So applying
the Doob transform with this choice of rescaling function amounts
to dividing every entry of $K$ by 3, then dividing the middle row
by 2 and multiplying the middle column by 2, giving 
\[
\global\long\def\arraystretch{1.5}
\hatk=\begin{bmatrix}\frac{1}{3} & \frac{2}{3} & 0\\
\frac{1}{6} & \frac{2}{3} & \frac{1}{6}\\
0 & \frac{2}{3} & \frac{1}{3}
\end{bmatrix}.
\]
This is a transition matrix as its rows sum to 1. Observe that $\hatk=\left[\frac{1}{3}\U\circ\D\right]_{\hatcalb_{3}}^{T}$,
where \begin{equation*} 
\hatcalb_3 :=\left\{ \raisebox{1.8ex}{\tableau{ \e & \e & \e}, \, $\displaystyle \frac{1}{2}$ \tableau{ \e & \e \\ \e}, \, \tableau{ \e \\ \e \\ \e}} \right\}.
\end{equation*} 

\textcolor{green}{}
\end{example}

\subsection{I.B: Quotient Spaces and Strong Lumping\label{sec:tableaux}}

As remarked in the introduction, sometimes only certain features of
a Markov chain is of interest - that is, we wish to study a process
$\{\theta(X_{t})\}$ rather than $\{X_{t}\}$, for some function $\theta$
on the state space. The process $\{\theta(X_{t})\}$ is called a lumping
(or projection), because it groups together states with the same image
under $\theta$, treating them as a single state. The analysis of
a lumping is easiest when $\{\theta(X_{t})\}$ is itself a Markov
chain. If this is true regardless of the initial state of the full
chain $\{X_{t}\}$, then the lumping is strong; if it is dependent
on the initial state, the lumping is weak. \cite[Sec. 6.3, 6.4]{lumping}
is a very thorough exposition on these topics.

This section focuses on strong lumping; the next section will handle
weak lumping.
\begin{defn}[Strong lumping]
 \label{def:stronglumping}Let $\{X_{t}\},\{\bar{X}_{t}\}$ be Markov
chains on state spaces $\Omega,\bar{\Omega}$ respectively, with transition
matrices $K,\bark.$ Then $\{\bar{X}_{t}\}$ is a \emph{strong lumping
of $\{X_{t}\}$ via $\theta$} if there is a surjection $\theta:\Omega\rightarrow\bar{\Omega}$
such that the process $\{\theta(X_{t})\}$ is a Markov chain with
transition matrix $\bark$, irrespective of the starting distribution
$X_{0}$. In this case, $\{X_{t}\}$ is a \emph{strong lift of $\{\bar{X}_{t}\}$
via $\theta$}.
\end{defn}
The following necessary and sufficient condition for strong lumping
is known as Dynkin's criterion: 
\begin{thm}[Strong lumping for Markov chains]
 \label{thm:stronglumping-general} \cite[Th. 6.3.2]{lumping} Let
$\{X_{t}\}$ be a Markov chain on a state space $\Omega$ with transition
matrix $K$, and let $\theta:\Omega\rightarrow\bar{\Omega}$ be a
surjection. Then $\{X_{t}\}$ has a strong lumping via $\theta$ if
and only if, for every $x_{1},x_{2}\in\Omega$ with $\theta(x_{1})=\theta(x_{2})$,
and every $\bar{y}\in\bar{\Omega}$, the transition probability sums
satisfy 
\[
\sum_{y:\theta(y)=\bary}K(x_{1},y)=\sum_{y:\theta(y)=\bary}K(x_{2},y).
\]
The lumped chain has transition matrix 
\[
\bark(\bar{x},\bar{y}):=\sum_{y:\theta(y)=\bary}K(x,y)
\]
 for any $x$ with $\theta(x)=\bar{x}$. \qed
\end{thm}
When the chain $\{X_{t}\}$ arises from linear operators via the Doob
$h$-transform, Dynkin's criterion translates into the statement below
regarding quotient operators.
\begin{thm}[Strong lumping for Markov chains from linear maps]
 \label{thm:stronglumping-linearmap} \cite[Th. 3.4.1]{mythesis}
Let $V$ be a vector space with basis $\calb$, and $\T:V\rightarrow V,\eta:V\rightarrow\mathbb{R}$
be linear maps allowing the Doob transform Markov chain construction
of Theorem \ref{thm:doob-transform}. Let $\barv$ be a quotient space
of $V$, and denote the quotient map by $\theta:V\rightarrow\barv$.
Suppose that 
\begin{enumerate}
\item the distinct elements of $\{\theta(x):x\in\calb\}$ are linearly independent,
and
\item $\T,\eta$ descend to maps on $\barv$ - that is, there exists $\bar{\T}:\barv\rightarrow\barv$,
$\bar{\eta}:\barv\rightarrow\mathbb{R}$, such that $\theta\circ\T=\bar{\T}\circ\theta$
and $\bar{\eta}\circ\theta=\eta$. 
\end{enumerate}

Then the Markov chain defined by $\bar{\T}$ (on the basis $\barcalb:=\{\theta(x):x\in\calb\}$,
with rescaling function $\bar{\eta}$) is a strong lumping via $\theta$
of the Markov chain defined by $\T$. 

\end{thm}
In the simplified case where $\eta\equiv1$ on $\calb$ and $\beta=1$
(so no rescaling is required to define the chain on $\calb$), such
as for random walks on groups, taking $\bar{\eta}\equiv1$ on $\barcalb$
satisfies $\bar{\eta}\circ\theta=\eta$, so condition 2 reduces to
a condition on $\T$ only, and the lumped chain also does not require
rescaling.

In the general case, the idea of the proof is that $\theta\circ\T=\bar{\T}\circ\theta$
is essentially equivalent to Dynkin's criterion for the unscaled matrices
$K:=[\T]_{\calb}^{T}$ and $\bark:=[\bar{\T}]_{\barcalb}^{T}$, and
this turns out to imply Dynkin's criterion for the Doob-transformed
transition matrices. 
\begin{proof}
Let $K=[\T]_{\calb}^{T}$, $\bark=[\bar{\T}]_{\barcalb}^{T}$, and
let $\beta$ be the eigenvalue of $\eta$. The first step is to show
that $\bar{\eta}$ is a possible rescaling function for $\bar{\T}$,
i.e. $\bar{\eta}$ is an eigenvector of $\bar{\T}^{*}$ with eigenvalue
$\beta$, taking positive values on $\barcalb$. In other words, the
requirement is that $[\bar{\T}^{*}(\bar{\eta})]v=\beta\bar{\eta}v$
for every $v\in\barv$, and $\bar{\eta}(v)>0$ if $v\in\barcalb$.
Since $\theta:V\rightarrow\barv$ and its restriction $\theta:\calb\rightarrow\barcalb$
are both surjective, it suffices to verify the above two conditions
for $v=\theta(x)$ with $x\in V$ and with $x\in\calb$ respectively. 

Now 
\[
[\bar{\T}^{*}(\bar{\eta})](\theta x)=\bar{\eta}\circ\bar{\T}(\theta x)=\bar{\eta}\circ\theta\circ\T(x)=\eta\circ\T(x)=\T^{*}\eta(x)=\beta\eta(x)=[\beta\bar{\eta}]\theta(x).
\]
And, for $x\in\calb$, we have $\bar{\eta}(\theta(x))=\eta(x)>0$.

Now let $\hatk,\hatbark$ denote the transition matrices that the
Doob transform constructs from $K$ and $\bark$. (Strictly speaking,
we do not yet know that the entries of $\bark$ are non-negative -
this will be proved in Equation (\ref{eq:stronglump}) below - but
the formula in the definition of the Doob transform remains well-defined
nevertheless.) By Theorem \ref{thm:stronglumping-general} above,
it suffices to show that, for any $x\in\calb$ with $\theta(x)=\barx$,
and any $\bary\in\barcalb$, 
\[
\hatbark(\bar{x},\bar{y})=\sum_{y:\theta(y)=\bary}\hatk(x,y).
\]
By definition of the Doob transform, this is equivalent to 
\[
\frac{1}{\beta}\bark(\bar{x},\bar{y})\frac{\bar{\eta}(\bary)}{\bar{\eta}(\barx)}=\frac{1}{\beta}\sum_{y:\theta(y)=\bary}K(x,y)\frac{\eta(y)}{\eta(x)}.
\]
Because $\bar{\eta}\theta=\eta$, the desired equality reduces to
\begin{equation}
\bark(\bar{x},\bar{y})=\sum_{y:\theta(y)=\bary}K(x,y).\label{eq:stronglump}
\end{equation}

Now expand both sides of $\bar{\T}\circ\theta(x)=\theta\circ\T(x)$
in the $\barcalb$ basis:
\[
\sum_{\bary\in\barcalb}\bark(\bar{x},\bar{y})\bary=\theta\left(\sum_{y\in\calb}K(x,y)y\right)=\sum_{\bary\in\barcalb}\left(\sum_{y:\theta(y)=\bary}K(x,y)\right)\bary.
\]
Equating coefficients of $\bary$ on both sides completes the proof.\end{proof}
\begin{example}[Forget the last bit under ``flip a random bit'']
 \label{ex:dbitbinaryIB} Take $G=\left(\mathbb{Z}/2\mathbb{Z}\right)^{d}$,
the additive group of binary strings of length $d$, as in Example
\ref{ex:dbitbinaryIA}. Then $\barg=\left(\mathbb{Z}/2\mathbb{Z}\right)^{d-1}$
is a quotient group of $G$, by forgetting the last bit. The quotient
map $\theta:G\rightarrow\barg$ induces a surjective map $\theta:\mathbb{R}G\rightarrow\mathbb{R}\barg$.

Recall that the ``flip a random bit'' chains comes from the linear
transformation on $\mathbb{R}G$ of right-multiplication by $q=\frac{1}{d}\left((1,0,\dots,0)+(0,1,0,\dots,0)+\dots+(0,\dots,0,1)\right)$.
Since multiplication of group elements descends to quotient groups,
forgetting the last bit is a lumping, and the lumped chain is associated
to right-multiplication in $\mathbb{R}\barg$ by the image of $q$
in $\mathbb{R}\barg$, which is $\frac{1}{d}\left((1,0,\dots,0)+(0,1,0,\dots,0)+\dots+(0,\dots,0,1)+(0,\dots,0)\right)$,
where there are $d$ summands each of length $d-1$.

The analogous construction holds for any quotient $\barg$ of any
group $G$; see \cite[App. IA]{riffleshufflerepeatedcards}.
\end{example}
The above principle extends to ``quotient sets'', i.e. a set of
coset representatives, which need not be groups (and is extended further
to double-coset representatives in \cite{bernoullilaplace}):
\begin{example}[``Follow the ace of spaces'' under shuffling]
 \label{ex:cardshuffleIB_aceofspades} Consider $G=\sn$, and let
$H=\snminusone$ be the subgroup of $\sn$ which permutes the last
$n-1$ objects. Then the transpositions $\tau_{i}:=(1\ i)$, for $2\leq i\leq n$,
together with $\tau_{1}:=\id$, give a set of right coset representatives
of $H$. The coset $H\tau_{i}$ consists of all deck orderings where
the card with label 1 is in the $i$th position from the top. Recall
that right-multiplication is always well defined on the set of right
cosets, so any card-shuffling model lumps by taking right cosets.
This corresponds to tracking only the location of the card with label
1. \cite[Sec. 2]{riffleshufflerepeatedcards} analyses this chain
in detail for the ``riffle-shuffles'' of \cite{originalriffleshuffle}.
\end{example}
The example of lumping to coset representatives can be further generalised
to a framework concerning orbits under group actions; notice in the
phrasing below that the theorem applies to more than random walks
on groups.
\begin{thm}[Strong lumping to orbits under group actions]
 \label{thm:stronglumping-groupaction} Let $V$ be a vector space
with basis $\calb$, and $\T:V\rightarrow V,\eta:V\rightarrow\mathbb{R}$
be linear maps allowing the Doob transform Markov chain construction
of Theorem \ref{thm:doob-transform}. Let $\{s_{i}:\calb\rightarrow\calb\}$
be a group of maps whose linear extensions to $V$ commute with $\T$
, and which satisfies $\eta\circ s_{i}=\eta$. Then the Markov chain
defined by $\T$ lumps to a chain on the $\{s_{i}\}$-orbits of $\calb$.
\end{thm}
This theorem recovers Example \ref{ex:cardshuffleIB_aceofspades}
above, of lumping a random walk on a group to right cosets, by letting
$s_{i}$ be left-multiplication by $i$, as $i$ ranges over the subgroup
$H$. Since $s_{i}$ is left-multiplication and $\T$ is right-multiplication,
they obviously commute. Hence any right-multiplication random walk
on a group lumps via taking right cosets.
\begin{proof}
Let $\barcalb$ be the sets of $\{s_{i}\}$-orbits of $\calb$, and
let $\theta:\calb\rightarrow\barcalb$ send an element of $\calb$
to its orbit. Let $\barv$ be the vector space spanned by $\barcalb$.
Then $\T$ descends to a well-defined map on $\barv$ because of the
following: if $\theta(x)=\theta(y)$, then $x=s_{i}(y)$ for some
$s_{i}$, so $\T(x)=\T\circ s_{i}(y)=s_{i}\circ\T(y)$ (using $s_{i}$
to denote the linear extension in this last expression), and so $\T(x)$
and $\T(y)$ are in the same orbit. And the condition $\eta\circ s_{i}=\eta$
ensures that $\bar{\eta}$ is well-defined on the $\{s_{i}\}$-orbits.
\end{proof}
Below are two more specialisations of Theorem \ref{thm:stronglumping-groupaction}
that hold for random walks on any group as long as the element being
multiplied is in the centre of the group algebra; we illustrate them
with card-shuffling examples. 
\begin{example}[Values of the top $k$ cards under random-transposition shuffling]
 \label{ex:cardshuffleIB_topcards} This simple example appears to
be new. Recall that the random-transposition shuffle corresponds to
right-multiplication by $q=\frac{1}{n}\id+\frac{2}{n}\sum_{i<j}(i\ j)$
on $\mathbb{R}\sn$. Because $q$ is a sum over all elements in two
conjugacy classes, it is in the centre of $\mathbb{R}\sn$, hence
right-multiplication by $q$ commutes with any other right-multiplication
operator. Let $s_{i}:\sn\rightarrow\sn$ be right-multiplication by
$i\in\snminusk$, the subgroup of $\sn$ which only permutes the last
$n-k$ objects. Then Theorem \ref{thm:stronglumping-groupaction}
implies that the random-transposition shuffle lumps to the orbits
under this action, which are the left-cosets of $\snminusk$ (see
also Example \ref{ex:cardshuffleIB_aceofspades}). The coset $\tau\snminusk$
consists of all decks whose top $k$ cards are $\tau(1),\tau(2),\dots,\tau(k)$
in that order (i.e. all decks that can be obtained from the identity
by first applying $\tau$ and then permuting the bottom $n-k$ cards
in any way). Hence the lumped chain tracks the values of the top $k$
cards.
\end{example}

\begin{example}[Coagulation-fragmentation]
 \label{ex:cardshuffleIB_ccl} As noted in \cite[Sec. 1.5]{randomwalkonmatchings},
the following chain is one specialisation of the processes in \cite{coagulationfragmentation},
modelling the splitting and recombining of molecules. 

Recall that the random-transposition shuffle corresponds to right-multiplication
by the central element $q=\frac{1}{n}\id+\frac{2}{n}\sum_{i<j}(i\ j)$
on $\mathbb{R}\sn$. Let $s_{i}:\sn\rightarrow\sn$ be conjugation
by the group element $i$, for all $i$ in $\sn$. This conjugation
action commutes with right-multiplication by $q$:
\[
s_{i}\circ T(x)=i(qx)i^{-1}=q(ixi^{-1})=T\circ s_{i}(x),
\]
where the second equality uses that $q$ is central. So Theorem \ref{thm:stronglumping-groupaction}
implies that the random-transposition shuffle lumps to the orbits
under this conjugation action, which are the conjugacy classes of
$\sn$. Each conjugacy class of $\sn$ consists precisely of the permutations
of a specific cycle type, and so can be labelled by the multiset of
cycle lengths, a partition of $n$ (see the start of Part II). These
cycle lengths represent the sizes of the molecules. As described in
\cite{coagulationfragmentationpersi}, right-multiplication by a transposition
either joins two cycles or breaks a cycle into two, corresponding
to the coagulation or fragmentation of molecules.
\end{example}
Our final example shows that the group $\{s_{i}\}$ inducing the lumping
of a random walk on $G$ need not be a subgroup of $G$:
\begin{example}[The Ehrenfest Urn]
 \label{ex:ehrenfesturn} \cite[Chap. 3.1.3]{randomwalksongroupspersibook}
Recall that the ``flip a random bit'' chain comes from the linear
transformation on $\mathbb{R}\left(\mathbb{Z}/2\mathbb{Z}\right)^{d}$
of right-multiplication by 
\[
q=\frac{1}{d}\left((1,0,\dots,0)+(0,1,0,\dots,0)+\dots+(0,\dots,0,1)\right).
\]
Let the symmetric group $\sd$ act on $\left(\mathbb{Z}/2\mathbb{Z}\right)^{d}$
by permuting the coordinates. Because $q$ is invariant under this
action, right-multiplication by $q$ commutes with this $\sd$ action.
So the ``flip a random bit'' Markov chain lumps to the $\sd$-orbits,
which track the number of ones in the binary string. The lumped walk
is as follows: if the current state has $k$ ones, remove a one with
probability $\frac{k}{d}$; otherwise add a one. As noted by \cite[Chap. 3.1.3]{randomwalksongroupspersibook},
interpreting the state of $k$ ones as $k$ balls in an urn and $d-k$
balls in another urn gives the classical Ehrenfest urn model: given
two urns containing $d$ balls in total, at each step, remove a ball
from either urn and place it in the other.\end{example}
\begin{rem*}
In the previous three examples, the action $\{s_{i}:G\rightarrow G\}$
respects multiplication on $G$:
\begin{equation}
s_{i}(gh)=s_{i}(g)s_{i}(h)\label{eq:actionrespectsmult}
\end{equation}
Then the random walk from right-multiplication by $q$ lumps to the
$\{s_{i}\}$-orbits if and only if $q$ is invariant under $\{s_{i}\}$.
But equation \ref{eq:actionrespectsmult} need not be true in all
applications of Theorem \ref{thm:stronglumping-groupaction} - see
Example \ref{ex:cardshuffleIB_aceofspades} where $s_{i}$ is left-multiplication.
\end{rem*}

\subsection{I.C: Subspaces and Weak Lumping\label{sec:permutations}}

Now turn to the weaker notion of lumping, where the initial distribution
matters.
\begin{defn}[Weak lumping]
 \label{def:weaklumping}Let $\{X_{t}\},\{X'_{t}\}$ be Markov chains
on state spaces $\Omega,\Omega'$ respectively, with transition matrices
$K,K'.$ Then $\{X'_{t}\}$ is a \emph{weak lumping of $\{X_{t}\}$
via $\theta$, with initial distribution $X_{0}$}, if there is a
surjection $\theta:\Omega\rightarrow\Omega'$ such that the process
$\{\theta(X_{t})\}$, started at the specified $X_{0}$, is a Markov
chain with transition matrix $K'$. In this case, $\{X_{t}\}$ is
a \emph{weak lift of $\{X'_{t}\}$ via $\theta$}.
\end{defn}
\cite[Th. 6.4.1]{lumping} gives a complicated necessary and sufficient
condition for weak lumping. (Note that they write $\pi$ for the initial
distribution and $\alpha$ for the stationary distribution.) Their
simple sufficient condition \cite[Th. 6.4.4]{lumping} has the drawback
of not identifying any valid initial distribution beyond the stationary
distribution - such a result would not be useful for the many descent
operator chains which are absorbing. So instead we appeal to a condition
for continuous Markov processes \cite[Th. 2]{weaklumping}, which
when specialised to the case of discrete time and finite state spaces
reads:
\begin{thm}[Sufficient condition for weak lumping for Markov chains]
 \label{thm:weaklumping-general} \cite[Th. 2]{weaklumping} Let
$K$ be the transition matrix of a Markov chain $\{X_{t}\}$ with
state space $\Omega$. Suppose $\Omega=\amalg\Omega^{i}$, and there
are distributions $\pi^{i}$ on $\Omega$ such that
\begin{enumerate}
\item $\pi^{i}$ is non-zero only on $\Omega^{i}$, 
\item The matrix 
\[
K'(i,j):=\sum_{x\in\Omega^{i},y\in\Omega^{j}}\pi^{i}(x)K(x,y)
\]
satisfies the equality of row vectors $\pi^{i}K=\sum_{j}K'(i,j)\pi^{j}$
for all $i$, or equivalently 
\[
\sum_{x\in\Omega}\pi^{i}(x)K(x,y)=\sum_{j}K'(i,j)\pi^{j}(y)
\]
for all $i$ and all $y$.
\end{enumerate}

Then, from any initial distribution of the form $\sum_{i}\alpha_{i}\pi^{i}$,
for constants $\alpha_{i}$, the chain $\{X_{t}\}$ lumps weakly to
the chain on the state space $\{\Omega^{i}\}$ with transition matrix
$K'$. \qed

\end{thm}
(This condition was implicitly used in \cite[Ex. 6.4.2]{lumping}.) 
\begin{rem*}
As the proof of Theorem \ref{thm:weaklumping-linearmap} will show,
the case of chains from the Doob transform without rescaling corresponds
to each $\pi^{i}$ being the uniform distribution on $\Omega^{i}$.
In this case, the conditions above simplify: we require that $K'(i,j):=\frac{1}{|\Omega^{i}|}\sum_{x\in\Omega^{i},y\in\Omega^{j}}K(x,y)$
satisfy $\frac{1}{|\Omega^{i}|}\sum_{x\in\Omega}K(x,y)=\frac{1}{|\Omega^{j}|}\sum_{j}K'(i,j)\pi^{j}(y)$
for all $i$ and all $y\in\Omega^{j}$. In other words, the only requirement
is that $\sum_{x\in\Omega}K(x,y)$ depends only on $\Omega^{j}\ni y$,
not on $y$, a condition somewhat dual to Doob's.
\end{rem*}
For Markov chains arising from the Doob transform, the condition $\pi^{i}K=\sum_{j}K'(i,j)\pi^{j}$
translates to the existence of invariant subspaces.  It may seem
strange to consider the subspace spanned by $\left\{ \sum_{x\in\calb^{i}}x\right\} $,
but Example \ref{ex:t2rshuffleIC} and Section II.C will show two
examples that arise naturally, namely permutation statistics and congruence
Hopf algebras.
\begin{thm}[Weak lumping for Markov chains from linear maps]
 \label{thm:weaklumping-linearmap} Let $V$ be a vector space with
basis $\calb$, and $\T:V\rightarrow V,\eta:V\rightarrow\mathbb{R}$
be linear maps admitting the Doob transform Markov chain construction
of Theorem \ref{thm:doob-transform}. Suppose $\calb=\amalg_{i}\calb^{i}$,
and write $x^{i}$ for $\sum_{x\in\calb^{i}}x$. Let $V'$ be the
subspace of $V$ spanned by the $x^{i}$, and suppose $\T(V')\subseteq V'$.
Define a map $\theta:\calb\rightarrow\{x^{i}\}$ by setting $\theta(x):=x^{i}$
if $x\in\calb^{i}$. Then the Markov chain defined by $\T:V\rightarrow V$
lumps weakly to the Markov chain defined by $\T:V'\rightarrow V'$
(with basis $\calb':=\{x^{i}\}$, and rescaling function the restriction
$\eta:V'\rightarrow\mathbb{R}$) via $\theta$, from any initial distribution
of the form $P\{X_{0}=x\}:=\alpha_{\theta(x)}\frac{\eta(x)}{\eta(\theta(x))}$,
where the $\alpha$s are constants depending only on $\theta(x)$.
In particular, if $\eta\equiv1$ on $\calb$ (so no rescaling is required
to define the chain on $\calb$), the Markov chain lumps from any
distribution which is constant on each $\calb^{i}$.
\end{thm}

Note that, in the simplified case $\eta\equiv1$, it is generally
not true that the restriction $\eta:V'\rightarrow\mathbb{R}$ is constant
on $\calb'$ - indeed, for $x^{i}\in\calb'$, it holds that $\eta(x^{i})=|\calb^{i}|$.
So a weak lumping chain from Theorem \ref{thm:weaklumping-linearmap}
will generally require rescaling.
\begin{rem*}
Suppose the conditions of Theorem \ref{thm:weaklumping-linearmap}
hold, and let $j:V'\hookrightarrow V$ be the inclusion map. Now the
dual map $j^{*}:V^{*}\twoheadrightarrow V'^{*}$, and $\T^{*}:V^{*}\rightarrow V^{*}$,
satisfy the hypotheses of Theorem \ref{thm:stronglumping-linearmap},
except that there may not be suitable rescaling functions $\eta:V^{*}\rightarrow\mathbb{R}$
and $\bar{\eta}:V'^{*}\rightarrow\mathbb{R}$. Because the Doob transform
chain for $\T^{*}$ is the time-reversal of the chain for $\T$ \cite[Th. 3.3.2]{mythesis},
this is a reflection of \cite[Th. 6.4.5]{lumping}.
\end{rem*}
The proof of Theorem \ref{thm:weaklumping-linearmap} is at the end
of this section.
\begin{example}[Number of rising sequences under riffle-shuffling]
 \label{ex:riffleshuffleIC} \cite[Cor. 2]{originalriffleshuffle}
The sequence $\{i,i+1,\dots,i+j\}$ is a rising sequence of a permutation
$\sigma$ if those numbers appear in that order when reading the one-line
notation of $\sigma$ from left to right. Viewing $\sigma$ as a deck
of cards, this says that the card with label $i$ is somewhere above
the card with label $i+1$, which is somewhere above the card with
label $i+2$, and so on, until the card with label $i+j$. Formally,
$\sigma^{-1}(i)<\sigma^{-1}(i+1)<\dots<\sigma^{-1}(i+j)$. Unless
otherwise specified, a rising sequence is assumed to be maximal, i.e.
$\sigma^{-1}(i-1)>\sigma^{-1}(i)<\sigma^{-1}(i+1)<\dots<\sigma^{-1}(i+j)>\sigma^{-1}(i+j+1)$. 

Following \cite{originalriffleshuffle}, write $R(\sigma)$ for the
number of rising sequences in $\sigma$. This statistic is also written
$\ides(\sigma)$, as it is the number of descents in $\sigma^{-1}$.
For example, $R(2,4,5,3,1)=3$, the three rising sequences being $\{1\}$,
$\{2,3\}$ and $\{4,5\}$.

\cite{originalriffleshuffle} studied the popular riffle-shuffle model,
where the deck is cut into two according to a binomial distribution
and interleaved. (We omit the details as this shuffle is not the focus
of the present paper). This arises from right-multiplication in $\mathbb{R}\sn$
by 
\[
q=\frac{n+1}{2^{n}}\id+\frac{1}{2^{n}}\sum_{R(\sigma)=2}\sigma.
\]
\cite[Cor. 2]{originalriffleshuffle} shows that riffle-shuffling,
if started from the identity, lumps weakly via the number of rising
sequences. This result can be slightly strengthened by applying the
present Theorem \ref{thm:weaklumping-linearmap} in conjunction with
\cite[Cor. 3]{originalriffleshuffle}, which proves explicitly that
$q$ generates a subalgebra spanned by $x^{i}:=\sum_{R(\sigma)=i}\sigma$,
for $1\leq i\leq n$. The authors recognised this subalgebra as equivalent
to Loday's ``number of descents'' subalgebra \cite{solomondescentalg}.
(More precisely: the linear extension to $\mathbb{R}\sn$ of the inversion
map $I(\sigma):=\sigma^{-1}$ is an algebra antimorphism - i.e. $I(\sigma\tau)=I(\tau)I(\sigma)$
- and it sends $x^{i}$ to the sum of permutations with $i-1$ descents,
which span Loday's algebra.) Since the basis elements $x^{i}$ have
the form stipulated in Theorem \ref{thm:weaklumping-linearmap}, it
follows that the lumping via the number of rising sequences is valid
starting from any distribution that is constant on the summands of
each $x^{i}$, i.e. on each subset of permutations with the same number
of rising sequences.
\end{example}
The key idea in the previous example is that, if $q\in\mathbb{R}G$
generates a subalgebra $V'$ of $\mathbb{R}G$ of the form described
in Theorem \ref{thm:weaklumping-linearmap}, then this produces a
weak lumping of the random walk on $G$ given by right-multiplication
by $q$. We apply this to the top-to-random shuffle:
\begin{example}[Length of last rising sequence under top-to-random shuffling]
 \label{ex:t2rshuffleIC} This simple example appears to be new.
In addition to the definitions in Example \ref{ex:riffleshuffleIC},
more terminology is necessary. The \emph{length} of the rising sequence
$\{i,i+1,\dots,i+j\}$ is $j+1$. Following \cite{cppriffleshuffle},
write $L(\sigma)$ for the length of the last rising sequence, meaning
the one which contains $n$. For example, the rising sequences of
$(2,4,3,5,1)$ have lengths $1,2,2$ respectively, and $L(2,4,3,5,1)=2$.

Recall that the top-to-random shuffle is given by right-multiplication
by 
\[
q=\frac{1}{n}\sum_{i=1}^{n}(i\ i-1\ \dots\ 1).
\]
The rising sequences of the cycles $(i\ i-1\ \dots\ 1)$ are precisely
$\{1\}$ and $\{2,3,\dots,n\}$, and these are the only permutations
$\sigma$ with $L(\sigma)=n-1$. \cite[Th. 4.2]{cppriffleshuffle}
shows that the algebra generated by $q$ is spanned by $x^{i}:=\sum_{L(\sigma)=i}\sigma$,
for $1\leq i\leq n$. Thus Theorem \ref{thm:weaklumping-linearmap}
shows that top-to-random shuffling weakly lumps via the length of
the last rising sequence, starting from any distribution that is constant
on permutations with the same last rising sequence length. In particular,
since the identity is the only permutation with $L(\sigma)=n$, the
lumping holds if the deck started at the identity permutation.\end{example}
\begin{proof}[Proof of Theorem \ref{thm:weaklumping-linearmap}]
In the notation of Theorem \ref{thm:weaklumping-general}, the distribution
$\pi^{i}$ is 
\[
\pi^{i}(x)=\begin{cases}
\frac{\eta(x)}{\eta(x^{i})} & \mbox{if }x\in\calb^{i};\\
0 & \mbox{otherwise},
\end{cases}
\]
which clearly satisfies condition 1.

To check condition 2, write $\T',\eta'$ for the restrictions of $\T,\eta$
to $V'$, and set $K=[\T]_{\calb}^{T}$, $K'=[\T']_{\calb'}^{T}$.
As in the proof of Theorem \ref{thm:stronglumping-linearmap}, we
start by showing that $\eta'$ is a possible rescaling function for
$\T'$: 
\[
(\T')^{*}(\eta')=\eta'\circ\T'=(\eta\circ\T)|_{V'}=(\beta\eta)|_{V'}=\beta\eta',
\]
so $\eta'$ is an eigenvector of $\T'^{*}$ with eigenvalue $\beta$.
And $\eta'$ is positive on $\calb'$ because $\eta'(x^{i})=\sum_{x\in\calb^{i}}\eta(x)$,
a sum of positive numbers. 

Write $\hatk,\hatk'$ for the associated transition matrices. (As
in the proof of Theorem \ref{thm:stronglumping-linearmap}, we check
that the entries of $K'$ are non-negative later, in Equation \ref{eq:weaklump}.)
We need to show that, for all $i$ and for all $y\in\calb$, 
\[
\sum_{x\in\calb^{i}}\pi^{i}(x)\hatk(x,y)=\sum_{j}\hatk'(x^{i},x^{j})\pi^{j}(y).
\]
Note that $\pi^{j}(y)$ is zero unless $y\in\calb^{j}$, so only one
summand contributes to the right hand side. By substituting for $\pi^{i},\hatk$
and $\hatk'$, the desired equality is equivalent to 
\[
\sum_{x\in\calb^{i}}\frac{\eta(x)}{\eta(x^{i})}\frac{1}{\beta}K(x,y)\frac{\eta(y)}{\eta(x)}=K'(x^{i},x^{j})\frac{1}{\beta}\frac{\eta(x^{j})}{\eta(x^{i})}\frac{\eta(y)}{\eta(x^{j})},
\]
which reduces to 
\begin{equation}
\sum_{x\in\calb^{i}}K(x,y)=K'(x^{i},x^{j})\label{eq:weaklump}
\end{equation}
 for $y\in\calb^{j}$.

Now, by expanding in the $\calb'$ basis, 
\[
\T'(x^{i})=\sum_{j}K'(x^{i},x^{j})x^{j}=\sum_{j}K'(x^{i},x^{j})\sum_{y\in\calb^{j}}y.
\]
On the other hand, a $\calb$ expansion yields
\[
\T'(x^{i})=\sum_{x\in\calb^{i}}\T(x)=\sum_{x\in\calb^{i}}\sum_{y\in\calb}K(x,y)y.
\]
So 
\[
\sum_{y\in\calb}\sum_{x\in\calb^{i}}K(x,y)y=\sum_{j}K'(x^{i},x^{j})\sum_{y\in\calb^{j}}y=\sum_{y}\sum_{j:y\in\calb^{j}}K'(x^{i},x^{j})y,
\]
and since $\calb$ is a basis, the coefficients of $y$ on the two
sides must be equal.
\end{proof}

\section{Part II: Lumpings from Subquotients of Combinatorial Hopf Algebras}

This part specialises the strong and weak lumping criteria of Part
I/Section 2 to Markov chains from descent operators on combinatorial
Hopf algebras \cite{descentoperatorchains}. Our main running example
(which in fact motivated the entire paper) is a lift for the ``down-up
chain on partitions'', where at each step a random box is removed
and a possibly different random box added, according to a certain
distribution (see the second half of Section II.A). The stationary
distribution of this chain is the Plancherel measure $\pi(\lambda)=\frac{(\dim\lambda)^{2}}{n!}$,
where $\dim\lambda$ is the dimension of the symmetric group representation
indexed by $\lambda$, or equivalently the number of standard tableaux
of shape $\lambda$ (see below for definitions). Because $(\dim\lambda)^{2}$
is the number of permutations whose RSK shape \cite[Sec. 7.11]{stanleyec2}\cite[Sec. 4]{fultonyoungtableaux}
is $\lambda$, it is natural to ask if the down-up chain on partitions
is the lumping of a chain on permutations with a uniform stationary
distribution.

Fulman \cite[Th. 3.1]{jasonlift} proved that this is almost true
for top-to-random shuffling: the probability distribution of the RSK
shape after $t$ top-to-random shuffles from the identity, agrees
with the probability distribution after $t$ steps of the down-up
chain on partitions. However, it is not true that top-to-random shuffles
lump via RSK shape. \cite[Fig. 1]{toptorandomshuffleboxmoves} is
an explicit 8-step trajectory of the partition chain that has no corresponding
trajectory in top-to-random shuffling. In other words, it is possible
for eight top-to-random shuffles to produce an RSK shape equal to
the end of the exhibited partition chain trajectory, but no choice
of intermediate steps will have RSK shapes equal to the given trajectory.
(Technically, this figure is written for random-to-top, the time-reversal
of top-to-random, so one should read it backwards from right to left,
apply it to top-to-random shuffles.)

The present Theorem \ref{thm:cpplift} finds that top-to-random shuffling
can be modified to give an honest weak lift of the down-up chain on
partitions: every time a card is moved, relabel it with the current
time that we moved the card, then track the (reversed) relative orders
of the labels. (A different interpretation without cards is in Section
\ref{sub:permutations}/II.C.) This lift is constructed in two stages
- Section \ref{sub:tableaux}/II.B builds a strong lift to tableaux
using Hopf algebra quotients, and Section \ref{sub:permutations}/II.C
builds a weak lift to permutations using Hopf subalgebras. Section
\ref{sec:startatid}/II.D then shows that the multistep transition
probabilities of the relabelled chain agree with the unmodified top-to-random
shuffle, if both are started at the identity, thus recovering the
Fulman result.

\subsection{Notation\label{sec:combinotation}}

A \emph{partition} $\lambda$ is a weakly-decreasing sequence of positive
integers: $\lambda:=(\lambda_{1},\dots,\lambda_{l})$ with $\lambda_{1}\geq\dots\geq\lambda_{l}>0$.
This is a \emph{partition of $n$}, denoted $\lambda\vdash n$, if
$\lambda_{1}+\dots+\lambda_{l}=n$. We will think of a partition $\lambda$
as a diagram of left-justified boxes with $\lambda_{1}$ boxes in
the topmost row, $\lambda_{2}$ boxes in the second row, etc. For
example, (5,2,2) is a partition of 9, and below is its diagram.

\noindent \begin{center} \begin{tabular}{c} \tableau{ \e & \e & \e & \e & \e \\ \e & \e \\ \e & \e } \end{tabular} \par\end{center}

A \emph{tableau of shape $\lambda$} is a filling of each of the boxes
in $\lambda$ with a positive integer. The \emph{shift} of a tableaux
$T$ by an integer $k$, denoted $T[k]$, increases each filling of
$T$ by $k$. A tableau is \emph{standard} if it is filled with $\{1,2,\dots,n\}$,
each integer occurring once. If no two boxes of a tableau $T$ has
the same filling, then its \emph{standardisation} $\std(T)$ is computed
by replacing the smallest filling by 1, the second smallest filling
by 2, and so on. Clearly $\std(T)$ is a standard tableau, of the
same shape as $T$. A box $b$ of $T$ is \emph{removable} if the
difference $T\backslash b$ is a tableau. Below shows a tableau of
shape $(5,2,2),$ its shift by 3, and its standardisation. The removable
boxes in the first tableau are $11$ and $13$.

\noindent \begin{center} \begin{tabular}{ccc} \tableau{ 1 & 2 & 5 & 10 & 13 \\ 4 & 8 \\ 6 & 11 } &  \tableau{ 4 & 5 & 8 & 13 & 16 \\ 7 & 11 \\ 9 & 14 } & \tableau{ 1 & 2 & 4 & 7 & 9 \\ 3 & 6 \\ 5 & 8 } \tabularnewline \tabularnewline $T$ & $T[3]$ & $\std(T)$ \tabularnewline \end{tabular} \par\end{center}

For a partition $\lambda$, write $\dim(\lambda)$ for the number
of standard tableaux of shape $\lambda$, as this is the dimension
of the symmetric group representation corresponding to $\lambda$
\cite[Chap. 2]{snreps}.

In the same vein, this paper will regard permutations as ``standard
words'', using one-line notation: $\sigma:=(\sigma(1),\dots,\sigma(n))$.
The \emph{length} of a word is its number of letters. The \emph{shift}
of a word $\sigma$ by an integer $k$, denoted $\sigma[k]$, increases
each letter of $\sigma$ by $k$. If a word $\sigma$ has all letters
distinct, then its \emph{standardisation} $\std(\sigma)$ is computed
by replacing the smallest letter by 1, the second smallest letter
by 2, and so on. Clearly $\std(\sigma)$ is a permutation. For example,
$\sigma=(6,1,4,8,2,11,10,13,5)$ is a word of length 9. Its shift
by 3 is $\sigma[3]=(9,4,7,11,5,14,13,16,8)$, and its standardisation
is $\std(\sigma)=(5,1,3,6,2,8,7,9,4)$.

We assume the reader is familiar with \emph{RSK insertion}, a map
from permutations to tableaux (only the insertion tableau is relevant
here, not the recording tableau), see \cite[Sec. 7.11]{stanleyec2}\cite[Sec. 4]{fultonyoungtableaux}.

A \emph{weak-composition} $D$ (also called a decomposition in \cite{hopfmonoidmarcelo})
is a list of non-negative integers $\left(d_{1},d_{2},\dots,d_{l(D)}\right)$.
This is a \emph{weak-composition of $n$}, denoted $D\vdash n$, if
$d_{1}+\dots+d_{l}=n$. For example, $(1,3,0,2,2,0,1)$ is a weak-composition
of 11.

A \emph{composition} $I$ is a list of positive integers $\left(i_{1},i_{2},\dots,i_{l(I)}\right)$,
where each $i_{k}$ is a \emph{part}. This is a \emph{composition
of $n$}, denoted $I\vdash n$, if $i_{1}+\dots+i_{l}=n$. For example,
$(1,3,2,2,1)$ is a composition of 11. Define a partial order on the
compositions of $n$: say $J\leq I$ if $J$ can be obtained by joining
adjacent parts of $I$. For example, $(6,2,1)\leq(1,3,2,2,1)$, and
also $(1,3,4,1)\leq(1,3,2,2,1)$.

The \emph{descent set} of a word $w=(w_{1},\dots,w_{n})$ is defined
to be $\left\{ j\in\{1,2,\dots,n-1\}|w_{j}>w_{j+1}\right\} $. It
is more convenient here to rewrite the descent set as a composition
in the following way: a word $w$ has \emph{descent composition} $\Des(w)=I$
if $i_{j}$ is the number of letters between the $j-1$th and $j$th
descent, i.e. if $w_{i_{1}+\dots+i_{j}}>w_{i_{1}+\dots+i_{j}+1}$
for all $j$, and $w_{r}\leq w_{r+1}$ for all $r\neq i_{1}+\dots+i_{j}$.
For example, the descent set of $(6,1,4,8,2,11,10,13,5)$ is $\{1,4,6,8\}$,
and $\Des(6,1,4,8,2,11,10,13,5)=(1,3,2,2,1)$. Note that $\Des(\sigma^{-1})$
consists of the lengths of the rising sequences (as in Example \ref{ex:riffleshuffleIC})
of $\sigma$.

\subsection{II.A: Markov Chains from Descent Operators, and the Down-Up Chain
on Partitions}

The Markov chains in this and subsequent sections arise from descent
operators on combinatorial Hopf algebras, through the framework of
\cite{descentoperatorchains} as summarised below.

Loosely speaking, a combinatorial Hopf algebra is a graded vector
space $\calh=\bigoplus_{n=0}^{\infty}\calh_{n}$ with a basis $\calb=\amalg_{n}\calbn$
indexed by a family of ``combinatorial objects'', such as partitions,
words, or permutations. The grading reflects the ``size'' of these
objects. $\calh$ admits a linear product map $m:\calh\otimes\calh\rightarrow\calh$
and a linear coproduct map $\Delta:\calh\rightarrow\calh\otimes\calh$
satisfying certain compatibility axioms; see the survey \cite{vicreinernotes}
for details. These two operations encode respectively how the combinatorial
objects combine and break. The concept was originally due to Joni
and Rota \cite{jonirota}, and the theory has since been expanded
in \cite{hivertcspolynomialrealisation,qsymisterminal,towersofalgs,hopfmonoids}
and countless other works. 

To define the descent operators, it is necessary to introduce a refinement
of the coproduct relative to the grading. Given a weak-composition
$D=\left(d_{1},d_{2},\dots,d_{l(D)}\right)$ of $n$, follow \cite{hopfmonoids}
and define $\Delta_{D}:\calh_{n}\rightarrow\calh_{d_{1}}\otimes\dots\otimes\calh_{d_{l(D)}}$
to be a projection to the graded subspace $\calh_{d_{1}}\otimes\dots\otimes\calh_{d_{l(D)}}$
of the iterated coproduct $(\Delta\otimes\id^{\otimes l(D)-1})\circ\dots\circ(\Delta\otimes\id\otimes\id)\circ(\Delta\otimes\id)\circ\Delta$.
So $\Delta_{D}$ models breaking an object into $l(D)$ pieces, of
sizes $d_{1},\dots,d_{l(D)}$ respectively. See the examples below.
\begin{example}
\label{ex:shufflealg,t2r} An instructive example of a combinatorial
Hopf algebra is the \emph{shuffle algebra} $\calsh$. Its basis is
the set of all words in the letters $\{1,2,\dots,N\}$ (for some $N$,
whose exact value is often unimportant). View the word $(w_{1},\dots,w_{n})$
as the deck of cards with card $w_{1}$ on top, card $w_{2}$ second
from the top, and so on, so card $w_{n}$ is at the bottom. The degree
of a word is its number of letters, i.e. the number of cards in the
deck. The product of two words, also denoted by $\shuffle$, is the
sum of all their interleavings (with multiplicity), and the coproduct
is deconcatenation, or cutting the deck. For example:
\begin{align*}
m((1,5)\otimes(5,2))=(1,5)\shuffle(5,2) & =2(1,5,5,2)+(1,5,2,5)+(5,1,5,2)+(5,1,2,5)+(5,2,1,5);\\
\Delta_{1,3}(1,5,5,2) & =(1)\otimes(5,5,2);\\
\Delta_{2,0,2}(1,5,5,2) & =(1,5)\otimes()\otimes(5,2).
\end{align*}
(Here, $()$ denotes the empty word, the unit of $\calsh$.) Observe
that 
\[
\frac{1}{4}m\circ\Delta_{1,3}(1,5,5,2)=\frac{1}{4}m((1)\otimes(5,5,2))=\frac{1}{4}(1,5,5,2)+\frac{1}{4}(5,1,5,2)+\frac{1}{4}(5,5,1,2)+\frac{1}{4}(5,5,2,1).
\]
The four words that appear on the right hand side are precisely all
the possible results after a top-to-random shuffle of the deck $(1,5,5,2)$,
and the coefficient of each word is the probability of obtaining it.
The same is true for decks of $n$ cards and the operator $\frac{1}{n}m\circ\Delta_{1,n-1}$.

Instead of removing only the top card - i.e. creating two piles of
sizes 1 and $n-1$ respectively - consider cutting the deck into $l$
piles of sizes $d_{1,}\dots,d_{l}$ for some weak-composition $D$
of $n$. Then interleave these $l$ piles together into one pile -
i.e. uniformly choose an ordering of all $n$ cards such that any
two cards from the same one of the $l$ piles stay in the same relative
order. Such a shuffle is described by (a suitable multiple of) the
composite operator $m\circ\Delta_{D}$. These composites (and their
linear combinations) are the \emph{descent operators} of \cite{descentoperators},
so named because, on a commutative or cocommutative Hopf algebra,
their composition is equivalent to the multiplication in Solomon's
descent algebra \cite{solomondescentalg} of the symmetric group.
This descent algebra view will be useful in the proof of Theorem \ref{thm:multistepprobofshuffles},
relating these shuffles to a different chain on permutations.

\cite{cppriffleshuffle} studied more general ``cut-and-interleave''
shuffles where the cut composition $D$ is random, according to some
probability distribution $P$ on the weak compositions of $n$. These
$P$-shuffles are described by 
\[
m\circ\Delta{}_{P}:=\sum_{D}\frac{P(D)}{\binom{n}{d_{1}\dots d_{l(D)}}}m\circ\Delta_{D},
\]
more precisely, their transition matrices are $\left[m\circ\Delta_{P}\right]_{\calbn}^{T}$,
where $\calbn$ is the word basis of the shuffle algebra. (The notation
$m\circ\Delta_{P}$, from \cite{descentoperatorchains}, is non-standard
and coined especially for this Markov chain application of descent
operators.) Notice that, if $P$ is concentrated at $(1,n-1)$, then
$m\circ\Delta_{P}=\frac{1}{n}m\circ\Delta_{1,n-1}$, corresponding
to the top-to-random shuffle as described in Example \ref{ex:shufflealg,t2r}.
The present paper will focus on the case where $P$ is concentrated
at $(n-1,1)$, so $m\circ\Delta_{P}=\frac{1}{n}m\circ\Delta_{n-1,1}$
models the ``bottom-to-random'' shuffle.

\cite{hopfpowerchains,descentoperatorchains} extend this idea to
other combinatorial Hopf algebras, using $m\circ\Delta_{P}$ to construct
Markov chains which model first breaking a combinatorial object into
$l$ pieces where the distribution of piece size is $P$, and then
reassembling the pieces. In particular, $\frac{1}{n}m\circ\Delta_{n-1,1}$
describes removing a piece of size 1 and reattaching it. For general
combinatorial Hopf algebras, this construction requires the Doob transform
(Theorem \ref{thm:doob-transform}).

To simplify the exposition, focus on the case where $|\calb_{1}|=1$,
i.e. there is only one combinatorial object of size 1. (This is not
true of the shuffle algebra, where $\calb_{1}$ consists of all the
different possible single card labels. Hence we will ignore the shuffle
algebra henceforth, until Section II.D.) Writing $\bullet$ for this
object, $\Delta_{1,\dots,1}(x)$ then is a multiple of $\bullet\otimes\dots\otimes\bullet=\bullet^{\otimes\deg x}$.
\cite[Lem 3.3]{descentoperatorchains} showed that, under the conditions
in Theorem \ref{thm:descentoperatorchain} below, this multiple is
a rescaling function. (Briefly, conditions i and ii guarantee that
$K:=[m\circ\Delta_{P}]$, the matrix before the Doob transform, has
non-negative entries, and condition iii ensures $\eta$ is positive
on $\calbn$.)\end{example}
\begin{thm}[Markov chains from descent operators]
 \label{thm:descentoperatorchain} \cite[Lem 3.3, Th. 3.4]{descentoperatorchains}
Suppose $\calh=\bigoplus_{n}\calhn$ is a graded connected Hopf algebra
with a basis $\calb=\amalg_{n}\calbn$ satisfying:\setcounter{enumi}{-1}
\begin{enumerate}[label=\roman*.]
\item $\calb_{1}=\{\bullet\}$;
\item for all $w,z\in\calb$, the expansion of $m(w\otimes z)$ in the $\calb$
basis has all coefficients non-negative;
\item for all $x\in\calb$, the expansion of $\Delta(x)$ in the $\calb\otimes\calb$
basis has all coefficients non-negative;
\item for all $x\in\calbn$ with $n>1$, it holds that $\Delta(x)\neq1\otimes x+x\otimes1$
(i.e. $\calbn$ contains no primitive elements when $n>1$).
\end{enumerate}

Then, for any fixed $n$ and any probability distribution $P(D)$
on weak-compositions $D$ of $n$, the corresponding descent operator
$m\circ\Delta{}_{P}:\calbn\rightarrow\calbn$ given by 
\[
m\circ\Delta{}_{P}:=\sum_{D}\frac{P(D)}{\binom{n}{d_{1}\dots d_{l(D)}}}m\circ\Delta_{D}
\]
and rescaling function $\eta:\calbn\rightarrow\mathbb{R}$ given by
\[
\eta(x):=\mbox{coefficient of }\bullet^{\otimes n}\mbox{ in }\Delta_{1,\dots,1}(x)
\]
admit the Doob transform construction of Theorem \ref{thm:doob-transform}. 

\end{thm}
The stationary distributions of these chains are easy to describe:
\begin{thm}
\label{thm:descentop_stationarydistribution} \cite[Th. 3.12]{descentoperatorchains}
The unique stationary distribution of the Markov chains constructed
in Theorem \ref{thm:descentoperatorchain} is
\[
\pi(x)=\frac{1}{n!}\eta(x)\times\mbox{coefficient of }x\mbox{ in }\bullet^{\otimes n},
\]
independent of the distribution $P$.
\end{thm}
\cite[Th. 3.5]{descentoperatorchains} derives the eigenvalues of
all descent operator chains. For our main example of $\frac{1}{n}m\circ\Delta_{n-1,1}$,
these eigenvalues are:
\begin{thm}
\label{thm:downup_evalues} \cite[Th. 4.4.i]{descentoperatorchains}
For the descent operator $\frac{1}{n}m\circ\Delta_{n-1,1}$, the eigenvalues
of the chains constructed in Theorem \ref{thm:descentoperatorchain}
are $\frac{j}{n}$ for $0\leq j\leq n$, $j\neq n-1$, and their multiplicities
are $\dim\calh_{n-j}-\dim\calh_{n-j-1}$.
\end{thm}
\cite[Th. 4.4]{descentoperatorchains} also describes some eigenvectors.
Since their formulae are complicated and they are not the focus of
the present paper, we do not go into detail here.

\subsubsection{Example: the Down-Up Chain on Partitions\label{sub:partitions}}

We explain in detail below the Markov chain that arises from applying
the Doob transform to $\frac{1}{n}m\circ\Delta_{n-1,1}$ on the algebra
of symmetric functions. This chain is one focus of \cite{jasonlift}.

Work with the algebra of symmetric functions $\Lambda$ \cite[Chap. 7]{stanleyec2},
with basis the Schur functions $\{s_{\lambda}\}$, which are indexed
by partitions. For clarity, we will often write $\lambda$ in place
of $s_{\lambda}$. The degree of $\lambda$ is the number of boxes
in its diagram.

As described in \cite[Sec. 2.5]{vicreinernotes}, $\Lambda$ carries
the following Hopf structure: 
\begin{eqnarray*}
m(s_{\nu}\otimes s_{\mu})=s_{\nu}s_{\mu} & = & \sum_{\lambda}c_{\nu\mu}^{\lambda}s_{\lambda};\\
\Delta(s_{\lambda}) & = & \sum_{\nu,\mu}c_{\nu\mu}^{\lambda}s_{\nu}\otimes s_{\mu},
\end{eqnarray*}
where $c_{\nu\mu}^{\lambda}$ are the Littlewood-Richardson coefficients.
These simplify greatly when $\mu$ is the partition $(1)$ - namely
$c_{\nu\mu}^{\lambda}=1$ if the diagrams of $\lambda$ and $\nu$
differ by one box, and $c_{\nu\mu}^{\lambda}=0$ otherwise. (This
is one case of the \emph{Pieri rule}.) Writing $\lambda\sim\nu\cup\square$
and $\nu\sim\lambda\backslash\square$ when the diagram of $\lambda$
can be obtained by adding one box to the diagram of $\nu$, the above
can be summarised as
\begin{eqnarray*}
m(\nu\otimes(1)) & = & \sum_{\lambda:\lambda\sim\nu\cup\square}\lambda;\\
\Delta_{\deg\nu-1,1}(\lambda) & = & \left(\sum_{\nu:\nu\sim\lambda\backslash\square}\nu\right)\otimes(1).
\end{eqnarray*}
For example, \begin{equation*} m \left( \raisebox{1.5ex}{\tableau{ \e & \e \\ \e \\ \e }} \otimes  \raisebox{1.5ex}{\tableau{ \e }}\right) = \raisebox{1.5ex}{ \tableau{ \e & \e & \e\\ \e \\ \e } } + \raisebox{1.5ex}{ \tableau{ \e & \e \\ \e & \e \\ \e } } + \raisebox{1.5ex}{ \tableau{ \e & \e \\ \e \\ \e \\ \e} }; \end{equation*} \begin{equation*} \Delta_{4,1} \left( \raisebox{1.5ex}{\tableau{ \e & \e & \e \\ \e \\ \e  } } \right) =\raisebox{1.5ex}{ \tableau{ \e & \e \\ \e \\ \e } } \otimes  \raisebox{1.5ex}{\tableau{ \e } }  +\raisebox{1.5ex}{ \tableau{ \e & \e & \e \\ \e } }   \otimes \raisebox{1.5ex}{\tableau{ \e } }  . \end{equation*}

To investigate how the Doob transform turns this data into probabilities,
it is necessary to first understand the rescaling function $\eta$.
By Theorem \ref{thm:descentoperatorchain}, $\eta(\lambda)$ is the
coefficient of $(1)^{\otimes\deg(\lambda)}$ in $\Delta_{1,\dots,1}(\lambda)$,
i.e. the number of ways to remove boxes one by one from $\lambda$.
Since such ways are in bijection with the standard tableaux of shape
$\lambda$, it holds that $\eta(\lambda)=\dim\lambda$. Hence the
Doob transform creates the following transition matrix from $\frac{1}{n}m\circ\Delta_{n-1,1}$:
\begin{eqnarray*}
\hatk(\lambda,\mu) & = & \sum_{\nu:\nu\sim\lambda\backslash\square,\mu\sim\nu\cup\square}\frac{1}{n}\frac{\dim\mu}{\dim\lambda}\\
 & = & \sum_{\nu:\nu\sim\lambda\backslash\square,\mu\sim\nu\cup\square}\frac{1}{n}\frac{\dim\mu}{\dim\nu}\frac{\dim\nu}{\dim\lambda}.
\end{eqnarray*}

The second expression suggests a decomposition of each time step into
two parts:
\begin{enumerate}
\item Remove a box from $\lambda$ to obtain $\nu$, with probability $\frac{\dim\nu}{\dim\lambda}$.
\item Add a box to $\nu$ to obtain $\mu$, with probability $\frac{1}{n}\frac{\dim\mu}{\dim\nu}$.
\end{enumerate}
The two parts are illustrated by down-right and up-right arrows respectively
in the following example trajectory in degree 5:\noindent 
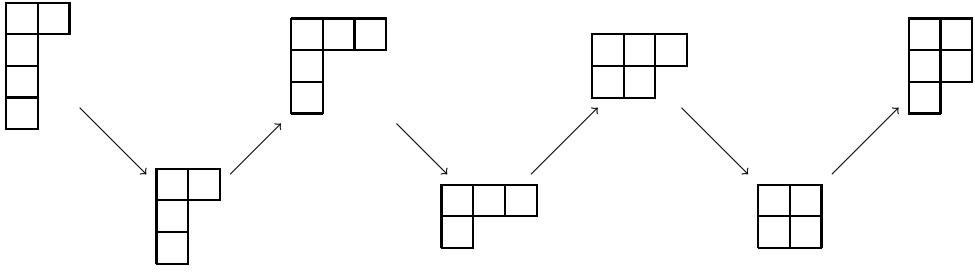
\begin{center} 
\begin{tikzpicture} 
\node (A) at (0,0) {\tableau{ \e & \e \\ \e \\ \e \\ \e }}; 
\node (B) at (4,0)  {\tableau{ \e & \e & \e \\ \e \\ \e  }}; 
\node (C) at (8,0)  {\tableau{ \e & \e & \e \\ \e & \e  }};  
\node (D) at (12,0) {\tableau{ \e & \e  \\ \e & \e \\ \e }}; 
\node (E) at (2,-2) {\tableau{ \e & \e \\ \e \\ \e }}; 
\node (F) at (6,-2)  {\tableau{ \e & \e & \e \\ \e }}; 
\node (G) at (10,-2)  {\tableau{ \e & \e \\ \e & \e  }}; 
\draw[->] (A) -- (E); 
\draw[->] (E) -- (B); 
\draw[->] (B) -- (F); 
\draw[->] (F) -- (C); 
\draw[->] (C) -- (G); 
\draw[->] (G) -- (D); 
\end{tikzpicture}
\par\end{center}

One easy way to implement step 1, the box removal, is via the \emph{hook
walk} of \cite{hookwalk}: uniformly choose a box $b$, then uniformly
choose a box in the \emph{hook} of $b$ - that is, to the right or
below $b$ - and continue uniformly picking from successive hooks
until you reach a removable box. Similarly, step 2 can be implemented
using the \emph{complimentary hook walk} of \cite{hookwalk2}: start
at the box (outside the partition diagram) in row $n$, column $n$,
uniformly choose a box in its \emph{complimentary hook} - that is,
to its left or above it, and outside of the partition - and continue
uniformly picking from complimentary hooks until you reach an addable
box.

The transition matrix of this chain in degree 3 is\[ \def\arraystretch{1.5} \begin{array}{c|ccc}  & (3) & (2,1) & (1,1,1)\\ \hline (3) & \frac{1}{3} & \frac{2}{3} & 0\\ (2,1) & \frac{1}{6} & \frac{2}{3} & \frac{1}{6}\\ (1,1,1) & 0 & \frac{2}{3} & \frac{1}{3} \end{array}. \]

To interpret the Markov chain on partitions from other descent operators
$m\circ\Delta_{P}$, it is necessary to view a partition of $n$ as
an irreducible representation of $\sn$, as explained in \cite[Chap. 2]{snreps}.
Then the multiplication and comultiplication of partitions come respectively
from the induction of the external product and the restriction to
Young subgroups - for irreducible representations corresponding to
the partitions $\mu\vdash i$, $\nu\vdash j$ and $\lambda\vdash n$,
\[
\mu\nu=\Ind_{\si\times\sj}^{\mathfrak{S}_{i+j}}\mu\times\nu;\quad\Delta_{n-i,i}(\lambda)=\Res_{\mathfrak{S}_{n-i}\otimes\si}^{\sn}\lambda.
\]
So the chains from $m\circ\Delta_{P}$ model restriction-then-induction,
as detailed below.
\begin{defn}
Each step of the \emph{$P$-restriction-then-induction} chain on irreducible
representations of the symmetric group $\sn$ goes as follows:
\begin{enumerate}
\item Choose a weak-composition $D=(d_{1},\dots,d_{l(D)})$ of $n$ with
probability $P(D)$.
\item Restrict the current irreducible representation to the chosen Young
subgroup $\mathfrak{S}_{d_{1}}\times\dots\times\mathfrak{S}_{d_{l(D)}}$.
\item Induce this representation to $\sn$, then pick an irreducible constituent
with probability proportional to the dimension of its isotypic component.
\end{enumerate}
\end{defn}
\cite{jasonlift} considered similar chains for subgroups of any group
$H\subseteq G$ instead of $\mathfrak{S}_{d_{1}}\times\dots\times\mathfrak{S}_{d_{l(D)}}\subseteq\sn$.

To calculate the common unique stationary distribution of all these
descent operator chains, using Theorem \ref{thm:descentop_stationarydistribution},
first note that, for $\lambda\vdash n$, 
\[
\mbox{coefficient of }\lambda\mbox{ in }(1)^{n}=|\{(\nu_{2},\dots,\nu_{n-1}):\nu_{2}\sim(1)\cup\square,v_{3}\sim\nu_{2}\cup\square,\dots,\lambda\sim\nu_{n-1}\cup\square\}|=\dim\lambda.
\]
To obtain $\pi(\lambda)$, multiply the above by $\frac{\eta(\lambda)}{n!}$.
As $\eta(\lambda)$ is also $\dim\lambda$, this means $\pi(\lambda)=\frac{(\dim\lambda)^{2}}{n!}$,
the Plancherel measure.

\subsection{II.B: Quotient Algebras and a Lift to Tableaux}

The following theorem is a specialisation of Theorem \ref{thm:stronglumping-linearmap},
about the strong lumping of Markov chains from linear maps, to the
case of descent operator chains.
\begin{thm}[Strong lumping for descent operator chains]
\label{thm:stronglumping-hopf} \cite[Th. 4.1]{cpmcfpsac} Let $\calh$,
$\barcalh$ be graded, connected Hopf algebras with bases $\calb$,
\textup{$\barcalb$} respectively, that both satisfy the conditions
in Theorem \ref{thm:descentoperatorchain}. If $\theta:\calh\rightarrow\barcalh$
is a Hopf-morphism such that $\theta(\calbn)=\barcalb_{n}$ for all
$n$, then the Markov chain on $\calbn$ which the Doob transform
fashions from the descent operator $m\circ\Delta_{P}$ lumps strongly
via $\theta$ to the Doob transform chain from the same operator on
$\barcalb_{n}$.\end{thm}
\begin{proof}
Condition 1 of Theorem \ref{thm:descentoperatorchain} requires the
distinct images of $\calbn$ under $\theta$ to be linearly independent
- this is true here by hypothesis.

Condition 2 requires $\theta\circ(m\circ\Delta_{P})=(m\circ\Delta_{P})\circ\theta$,
and $\bar{\eta}\circ\theta=\eta$. The former is true because $\theta$
is a Hopf-morphism. To check the latter, apply both sides to an arbitrary
$x\in\calb_{n}$ and multiply by $\bar{\bullet}^{\otimes n}$, where
$\bar{\bullet}$ is the unique element of $\barcalb_{1}$; then the
condition required is equivalent to 
\begin{equation}
\bar{\eta}(\theta(x))\bar{\bullet}^{\otimes n}=\eta(x)\bar{\bullet}^{\otimes n}.\label{eq:stronglumpeta}
\end{equation}
The left hand side is $\Delta_{1,\dots,1}(\theta(x))$, by definition
of $\bar{\eta}$. Because $\theta$ is a Hopf-morphism, this is equal
to $(\theta\otimes\dots\otimes\theta)\circ\Delta_{1,\dots,1}(x)=(\theta\otimes\dots\otimes\theta)\left(\eta(x)\bullet^{\otimes n}\right)$,
by definition of $\eta$ (writing $\bullet$ for the unique element
of $\calb_{1}$). Hence this is $(\theta(\bullet)\otimes\dots\otimes\theta(\bullet))\eta(x)$.
Since $\theta(\calbn)=\barcalbn$ for all $n$, it is true for $n=1$,
whcih means $\theta(\bullet)\in\bar{\calb}_{1}$. Since $\bar{\calb}_{1}=\{\bar{\bullet}\}$,
it must be that $\theta(\bullet)=\bar{\bullet}$. Hence $(\theta(\bullet)\otimes\dots\otimes\theta(\bullet))\eta(x)=\eta(x)\bar{\bullet}^{\otimes n}$,
proving Equation \ref{eq:stronglumpeta}.\end{proof}
\begin{rem*}
Observe that the proof does not fully use the assumption $\theta(\calbn)=\barcalbn$
for all $n$ - all that is required is that $\theta(\calbn)=\barcalbn$
for the single value of $n$ of interest, and that $\theta(\calb_{1})=\barcalb_{1}$.
Indeed, if $\calbn$ can be partitioned into communication classes
$\calbn=\amalg_{i}\calbn^{(i)}$ for the $m\circ\Delta_{P}$ Markov
chain (i.e. it is impossible to move between distinct $\calbn^{(i)}$
using the $m\circ\Delta_{P}$ Markov chain), so there is effectively
a separate chain on each $\calbn^{(i)}$, then, to prove a lumping
for the chain on one $\calbn^{(i)}$, it suffices to require $\theta(x)\in\barcalbn$
only for $x\in\calbn^{(i)}$ (and $\theta(\calb_{1})=\barcalb_{1}$).
This will be useful in Section \ref{sub:descentsetlump}/II.E for
showing that cut-and-interleave shuffles of $n$ distinct cards lump
via descent set, as this lumping is false for non-distinct decks.
\end{rem*}

\subsubsection{Example: the Down-Up Chain on Standard Tableaux\label{sub:tableaux}}

To use Theorem \ref{thm:stronglumping-hopf} to lift the descent operator
chains on partitions of the previous section, we need a Hopf algebra
whose quotient is $\Lambda$, and the quotient map must come from
a map from the basis elements of the new, larger Hopf algebra to partitions
(or more accurately, to Schur functions). Below describes one such
algebra, the Poirer-Reutenauer Hopf algebra of standard tableaux.
It was christened $(\mathbb{Z}T,*,\delta)$ in \cite{fsym}, but we
follow \cite[Sec. 3.5]{sym6} and denote it by $\fsym$, for ``free
symmetric functions''. Its distinguished basis is $\{\mathbf{S}_{T}\}$,
where $T$ runs over the set of standard tableaux. As with partitions,
it will be convenient to write $T$ in place of $\mathbf{S}_{T}$.
This algebra is graded by the number of boxes in $T$. The quotient
map $\fsym\rightarrow\Lambda$ is essentially taking the shape of
the standard tableaux - the image of $\mathbf{S}_{T}$ in $\Lambda$
is $s_{\sh(T)}$.

Because the product and coproduct of $\fsym$ are fairly complicated,
involving Jeu de Taquin and other tableaux manipulations, we describe
here only $m:\calh_{n-1}\otimes\calh_{1}\rightarrow\calhn$ and $\Delta_{n-1,1}$,
and direct the interested reader to \cite[Sec. 5c, 5d]{fsym} for
details.

If $T$ is a standard tableaux with $n-1$ boxes, then the product
$m(T\otimes\tableau{1})$ is the sum of all ways to add a new box,
filled with $n$, to $T$. For example, \begin{equation*} m \left( \raisebox{1.5ex}{\tableau{ 1 & 2 \\ 3 \\ 4 } } \otimes  \raisebox{1.5ex}{\tableau{ 1 } } \right)= \raisebox{1.5ex}{ \tableau{ 1 & 2 & 5\\ 3 \\ 4 } } + \raisebox{1.5ex}{ \tableau{ 1 & 2 \\ 3 & 5 \\ 4 } } + \raisebox{1.5ex}{ \tableau{ 1 & 2 \\ 3 \\ 4 \\ 5} }. \end{equation*}

The coproduct $\Delta_{n-1,1}$ is ``unbump and standardise''. \cite[fourth paragraph of proof of Th. 7.11.5]{stanleyec2}
explains unbumping as follows: for a removable box $b$ in row $i$,
remove $b$, then find the box in row $i-1$ containing the largest
integer smaller than $b$. Call this filling $b_{1}$. Replace $b_{1}$
with $b$, then put $b_{1}$ in the box in row $i-2$ previously filled
with the largest integer smaller than $b_{1}$, and continue this
process up the rows. In the second term in the example below, these
displaced fillings are $4,3,2$. What unbumping achieves is this:
if $b$ was the last number to be inserted in an RSK insertion that
resulted in $T$, then unbumping $b$ from $T$ recovers the tableaux
before $b$ was inserted. The coproduct $\Delta_{n-1,1}(T)$ is the
sum of unbumpings over all removable boxes $b$ of $T$, then standardising
the unbumped tableaux, for example  \begin{align*} \Delta_{4,1} \left( \raisebox{1.5ex}{\tableau{ 1 & 2 & 5 \\ 3 \\ 4  } } \right) &  = \std \left(\raisebox{1.5ex}{ \tableau{ 1 & 2 \\ 3 \\ 4 } } \right) \otimes \tableau{ 1 }  + \std \left( \raisebox{1.5ex}{ \tableau{ 1 & 3 & 5 \\ 4 } }  \right) \otimes \tableau{ 1 }  \\ &  = \mbox{ \tableau{ 1 & 2 \\ 3 \\ 4 } }  \otimes \tableau{ 1 }   +  \mbox{ \tableau{ 1 & 2 & 4 \\ 3 } }   \otimes \tableau{ 1 } . \end{align*}

To describe the down-up chain on standard tableaux (i.e. the chain
which the Doob transform fashions from the map $\frac{1}{n}m\circ\Delta_{n-1,1}$),
it remains to calculate the rescaling function $\eta(T)$. This is
the coefficient of $\tableau{1}^{\otimes n}$ in $\Delta_{1,\dots,1}(T)$,
which the description of $\Delta_{n-1,1}$ above rephrases as the
number of ways to successively choose boxes to unbump from $T$. Such
ways are in bijection with the standard tableaux of the same shape
as $T$, so $\eta(T)=\dim(\sh T)$. Hence one step of the down-up
chain on standard tableaux, starting from a tableau $T$ of $n$ boxes,
has the following interpretation:
\begin{enumerate}
\item Pick a removable box $b$ of $T$ with probability $\frac{\dim(\sh(T\backslash b))}{\dim(\sh T)}$,
and unbump $b$. (As for partitions, one can pick $b$ using the hook
walk of \cite{hookwalk}.)
\item Standardise the remaining tableaux and call this $T'$.
\item Add a box labelled $n$ to $T'$, with probability $\frac{1}{n}\frac{\dim(\sh(T'\cup n))}{\dim(\sh T')}$.
(As for partitions, one can pick where to add this box using the complimentary
hook walk of \cite{hookwalk2}.)
\end{enumerate}
Here are a few steps of a possible trajectory in degree 5 (the red
marks the unbumping paths):

\noindent 
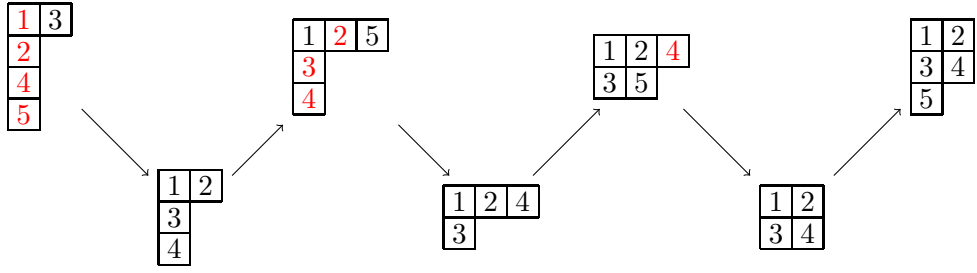
\begin{center} 
\begin{tikzpicture} 
\node (A) at (0,0) {\tableau{ \color{red}{1} & 3 \\ \color{red}{2} \\ \color{red}{4} \\ \color{red}{5} }}; 
\node (B) at (4,0)  {\tableau{ 1 & \color{red}{2} & 5 \\ \color{red}{3} \\ \color{red}{4}  }}; 
\node (C) at (8,0)  {\tableau{ 1 & 2 & \color{red}{4} \\ 3 & 5  }};  
\node (D) at (12,0) {\tableau{ 1 & 2  \\ 3 & 4 \\ 5 }}; 
\node (E) at (2,-2) {\tableau{ 1 & 2 \\ 3 \\ 4 }}; 
\node (F) at (6,-2)  {\tableau{ 1 & 2 & 4 \\ 3 }}; 
\node (G) at (10,-2)  {\tableau{ 1 & 2 \\ 3 & 4  }}; 
\draw[->] (A) -- (E); 
\draw[->] (E) -- (B); 
\draw[->] (B) -- (F); 
\draw[->] (F) -- (C); 
\draw[->] (C) -- (G); 
\draw[->] (G) -- (D); 
\end{tikzpicture}
\par\end{center}

The transition matrix of this chain in degree 3 is\[ \arraycolsep=15pt \def\arraystretch{2} \begin{array}{c|cccc}  & \tableau{1 & 2 & 3} & \tableau{1 & 2 \\ 3} & \tableau{1 & 3 \\ 2} & \tableau{1 \\ 2 \\ 3}\\ \hline \tableau{1 & 2 & 3} & \frac{1}{3} & \frac{2}{3} & 0 & 0\\ \tableau{1 & 2 \\ 3} & \frac{1}{6} & \frac{1}{3} & \frac{1}{3} & \frac{1}{6}\\ \tableau{1 & 3 \\ 2} & \frac{1}{6} & \frac{1}{3} & \frac{1}{3} & \frac{1}{6}\\ \tableau{1 \\ 2 \\ 3} &  0&  0& \frac{2}{3} & \frac{1}{3} \end{array}. \]

According to Proposition \ref{thm:descentop_stationarydistribution},
the unique stationary distribution of the down-up chain on tableau
is $\pi(T)=\frac{1}{n!}\eta(T)\times$coefficient of $T$ in $\tableau{1}^{n}$.
Note that there is a unique way of adding outer boxes filled with
$1,2,\dots$ in succession to build a given tableau $T$, so each
tableau of $n$ boxes appears precisely once in the product $\tableau{1}^{n}$.
Hence $\pi(T)=\frac{1}{n!}\eta(T)=\frac{1}{n!}\dim(\sh T)$.

As explained at the beginning of this subsection, the symmetric functions
$\Lambda$ is a quotient of $\fsym$ \cite[Th. 4.3.i]{fsym}, and
the quotient map sends $\mathbf{S}_{T}$ to $s_{\sh(T)}$. Applying
Theorem \ref{thm:stronglumping-hopf} then gives:
\begin{thm}
\label{thm:tableauxtopartition} The down-up Markov chain on standard
tableaux lumps to the down-up Markov chain on partitions via taking
the shape. \qed
\end{thm}
By the same argument, the $P$-restriction-then-induction chains on
partitions, for any probability distribution $P$, lift to $m\circ\Delta_{P}$
chains on tableaux, but these are hard to describe.

\subsection{II.C: Subalgebras and a Lift to Permutations}

The following theorem is a specialisation of Theorem \ref{thm:weaklumping-linearmap},
about the weak lumping of Markov chains from linear maps, to the case
of descent operator chains.
\begin{thm}[Weak lumping for descent operator chains]
\label{thm:weaklumping-hopf} Let $\calh$, $\calh'$ be graded,
connected Hopf algebras with state space bases $\calb$, \textup{$\calb'$}
respectively, that both satisfy the conditions in Theorem \ref{thm:descentoperatorchain}.
Suppose for all $n$ that $\theta:\calbn\rightarrow\calbn'$ is such
that the ``preimage sum'' map $\theta^{*}:\calbn'\rightarrow\calhn$,
defined by $\theta^{*}(x'):=\sum_{x\in\calb,\theta(x)=x'}x$, extends
to a Hopf-morphism. Then the Markov chain on $\calbn$ which the Doob
transform fashions from the descent operator $m\circ\Delta_{P}$ lumps
weakly via $\theta$ to the Doob transform chain from the same map
on $\calbn'$, from any starting distribution $X_{0}$ where $\frac{X_{0}(x)}{\eta(x)}=\frac{X_{0}(y)}{\eta(y)}$
whenever $\theta(x)=\theta(y)$.\end{thm}
\begin{proof}
First observe that $\theta^{*}$ sends $\calbn'$ to a linearly independent
set in $\calhn$, so $\theta^{*}:\calh'\rightarrow\calh$ is injective,
hence it is legal to identify $x'\in\calh'$ with $\sum_{x\in\calb,\theta(x)=x'}x$
and view $\calh'$ as a Hopf subalgebra of $\calh$. So $\calh'$
is an invariant subspace of $\calh$ under $m\circ\Delta_{P}$. The
other requirements of Theorem \ref{thm:weaklumping-linearmap} are
that each element in the basis $\calbn'$ should be a sum over disjoint
subsets of $\calbn$, which is true by hypothesis; and that the restriction
of the rescaling function $\eta:\calhn\rightarrow\mathbb{R}$ to $\calbn'$
is the natural rescaling function for $\calhn'$. The definition of
$\eta'$ is that, for all $x\in\calbn'$, it holds that $\Delta_{1,\dots,1}(x')=\eta'(x')\bullet'^{\otimes n}$,
where $\bullet'$ is the unique element of $\calb_{1}'$. Since $\theta$
sends $\calb_{1}$ to $\calb_{1}'$, it must be true that $\bullet'=\theta(\bullet)$,
i.e. $\bullet=\theta^{*}(\bullet')$, or $\bullet=\bullet'$ when
viewing $\calh'$ as a subalgebra of $\calh$. Hence $\Delta_{1,\dots,1}(x')=\eta'(x')\bullet{}^{\otimes n}$.
The left hand side is 
\[
\Delta_{1,\dots,1}\left(\sum_{x\in\calb,\theta(x)=x'}x\right)=\sum_{x\in\calb,\theta(x)=x'}\Delta_{1,\dots,1}(x)=\sum_{x\in\calb,\theta(x)=x'}\eta(x)\bullet^{\otimes n}=\eta(x')\bullet^{\otimes n},
\]
where the second equality uses the definition of $\eta$, and the
third equality uses the linearity of $\eta$. Hence $\eta'(x')=\eta(x')$.
\end{proof}

\subsubsection{Example: the Down-Up Chain on Permutations\label{sub:permutations}}

Recall that the previous section lifted the down-up chain on partitions
to the down-up chain on tableaux, using that the symmetric functions
$\Lambda$ is a quotient of the Hopf algebra $\fsym$ of tableaux.
To further lift this chain to permutations using Theorem \ref{thm:weaklumping-hopf},
we express $\fsym$ as a subalgebra of a Hopf algebra of permutations,
namely the Malvenuto-Reutenauer algebra.

Before a full description of this Hopf structure on permutations,
here is an intuitive interpretation of its down-up chain, from $\frac{1}{n}m\circ\Delta_{n-1,1}$.
(This is a mild variant of \cite[Ex. 1.2]{descentoperatorchains},
which is for $\frac{1}{n}m\circ\Delta_{1,n-1}$.) You keep an electronic
to-do list of $n$ tasks. Each day, you complete the task at the top
of the list, and are handed a new task, which you add to the list
in a position depending on its urgency (more urgent tasks are placed
closer to the top). Assume the incoming tasks are equally distributed
in urgency, so they are each inserted into the list in a uniform position.
To assign a permutation (in one-line notation) to each daily list,
first write $i$ for the task received on day $i$ (for $i\geq n$),
then standardise, and read these numbers from the bottom of the list
to the top. (After the standardisation, the numbers will indicate
the relative time that each task has spent on the list: $1$ denotes
the task that's been on the list for the longest time, $2$ for the
next oldest task, and so on, so $n$ denotes the task you received
today.) The diagram below shows one possibility over four days of
both the original task numbering before standardisation (vertically)
and the Markov chain on permutations from the standardised task numbering
(horizontally):

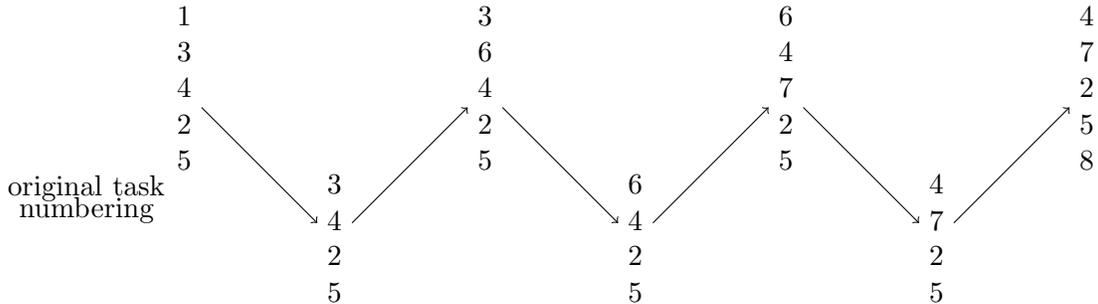
\def\arraystretch{1} \noindent \begin{center}
\begin{tikzpicture} 
\node (X) at (-1.3, -1.3) {original task}; 
\node (Y) at (-1.3, -1.6) {numbering};
\node (A) at (0,0) {$\begin{matrix}1\\3\\4\\2\\5\end{matrix}$}; 
\node (B) at (4,0)  {$\begin{matrix}3\\6\\4\\2\\5\end{matrix}$}; 
\node (C) at (8,0)  {$\begin{matrix}6\\4\\7\\2\\5\end{matrix}$};  
\node (D) at (12,0) {$\begin{matrix}4\\7\\2\\5\\8\end{matrix}$}; 
\node (E) at (2,-2) {$\begin{matrix}3\\4\\2\\5\end{matrix}$}; 
\node (F) at (6,-2)  {$\begin{matrix}6\\4\\2\\5\end{matrix}$}; 
\node (G) at (10,-2)  {$\begin{matrix}4\\7\\2\\5\end{matrix}$}; 
\draw[->] (A) -- (E); 
\draw[->] (E) -- (B); 
\draw[->] (B) -- (F); 
\draw[->] (F) -- (C); 
\draw[->] (C) -- (G); 
\draw[->] (G) -- (D); 
\end{tikzpicture}
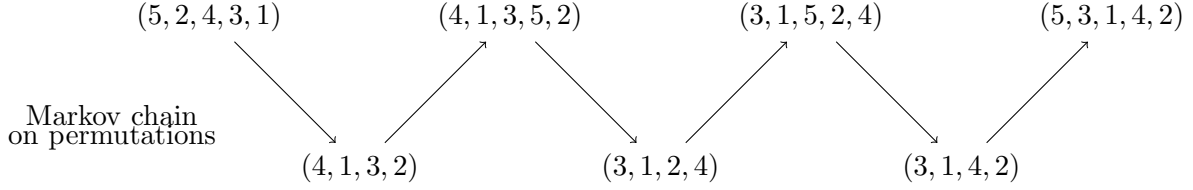
\begin{tikzpicture} 
\node (X) at (-1.3, -1.3) {Markov chain}; 
\node (Y) at (-1.3, -1.6) {on permutations};
\node (A) at (0,0) {$(5,2,4,3,1)$}; 
\node (B) at (4,0)  {$(4,1,3,5,2)$}; 
\node (C) at (8,0)  {$(3,1,5,2,4)$};  
\node (D) at (12,0) {$(5,3,1,4,2)$}; 
\node (E) at (2,-2) {$(4,1,3,2)$}; 
\node (F) at (6,-2)  {$(3,1,2,4)$}; 
\node (G) at (10,-2)  {$(3,1,4,2)$}; 
\draw[->] (A) -- (E); 
\draw[->] (E) -- (B); 
\draw[->] (B) -- (F); 
\draw[->] (F) -- (C); 
\draw[->] (C) -- (G); 
\draw[->] (G) -- (D); 
\end{tikzpicture}
\par\end{center}

The same chain arises from performing the top-to-random shuffle and
keeping track of the relative last times that the cards were last
touched, instead of their values.

As mentioned above, this is the chain associated to $\frac{1}{n}m\circ\Delta_{n-1,1}$
on the Malvenuto-Reutenauer Hopf algebra of permutations, denoted
$(\mathbb{Z}S,*,\Delta)$ in \cite[Sec. 3]{fqsym}, $(\mathbb{Z}S,*,\delta)$
in \cite{fsym}, and $\mathfrak{S}Sym$ in \cite{fqsymstructure}.
We follow the recent Parisian literature, such as \cite{sym6}, and
call this algebra $\fqsym$, for ``free quasisymmetric functions''. 

The basis of concern here is the fundamental basis $\{\mathbf{F}_{\sigma}\}$,
as $\sigma$ ranges over all permutations (of any length). As in the
previous sections, we often write $\mathbf{F}_{\sigma}$ simply as
$\sigma$. The degree of $\sigma$ is its length when considered as
a word. 

We explain the Hopf structure on $\fqsym$ by example. The product
$\sigma_{1}\sigma_{2}$ is $\sigma_{1}\shuffle\sigma_{2}[\deg\sigma_{1}]$
, the sum of all ``interleavings'' or ``shuffles'' of $\sigma_{1}$
with the shift of $\sigma_{2}$ by $\deg(\sigma_{1})$: 
\begin{align*}
(3,1,2)(2,1) & =(3,1,2)\shuffle(5,4)\\
 & =(3,1,2,5,4)+(3,1,5,2,4)+(3,1,5,4,2)+(3,5,1,2,4)+(3,5,1,4,2)\\
 & \hphantom{=}+(3,5,4,1,2)+(5,3,1,2,4)+(5,3,1,4,2)+(5,3,4,1,2)+(5,4,3,1,2).
\end{align*}

The coproduct is ``deconcatenate and standardise'': 
\[
\Delta(\sigma)=\sum_{\sigma_{1}\cdot\sigma_{2}=\sigma}\std(\sigma_{1})\otimes\std(\sigma_{2}),
\]
where $\cdot$ denotes concatenation. Thus
\begin{align*}
 & \hphantom{=}\Delta(4,1,3,2)\\
 & =()\otimes(4,1,3,2)+\std(4)\otimes\std(1,3,2)+\std(4,1)\otimes\std(3,2)+\std(4,1,3)\otimes\std(2)+(4,1,3,2)\otimes()\\
 & =()\otimes(4,1,3,2)+(1)\otimes(1,3,2)+(2,1)\otimes(2,1)+(3,1,2)\otimes(1)+(4,1,3,2)\otimes().
\end{align*}

Recall that we are primarily interested in $\frac{1}{n}m\circ\Delta_{n-1,1}$.
Note that $\Delta_{n-1,1}$ removes the last letter of the word and
standardises the result, whilst right-multiplication by $(1)$ yields
the sum of all ways to insert the letter $n$. Since $\Delta_{n-1,1}(\sigma)$
contains only one term, we see inductively that the rescaling function
is $\eta(\sigma)\equiv1$. Hence one step of the down-up chain on
$\fqsym$, starting at $\sigma\in\sn$, has the following description:
\begin{enumerate}
\item Remove the last letter of $\sigma$.
\item Standardise the remaining word.
\item Insert the letter $n$ into this standardised word, in a uniformly
chosen position.
\end{enumerate}
The transition matrix of this chain in degree 3 is (all empty entries
are 0) \[  \def\arraystretch{1.2} \begin{array}{c|cccccc}  & (1,2,3) & (1,3,2) & (3,1,2) & (2,3,1) & (2,1,3) & (3,2,1)\\ \hline (1,2,3) & \frac{1}{3} & \frac{1}{3} & \frac{1}{3}\\ (1,3,2) & \frac{1}{3} & \frac{1}{3} & \frac{1}{3}\\ (3,1,2) &  &  &  & \frac{1}{3} & \frac{1}{3} & \frac{1}{3}\\ (2,3,1) & \frac{1}{3} & \frac{1}{3} & \frac{1}{3}\\ (2,1,3) &  &  &  & \frac{1}{3} & \frac{1}{3} & \frac{1}{3}\\ (3,2,1) &  &  &  & \frac{1}{3} & \frac{1}{3} & \frac{1}{3} \end{array}. \]

According to Proposition \ref{thm:descentop_stationarydistribution},
the unique stationary distribution of this chain is $\pi(\sigma)=\frac{1}{n!}\eta(\sigma)\times$coefficient
of $\sigma$ in $(1)^{n}$. Since there is a unique way of inserting
the letters $1,2,\dots$ in that order to obtain a given permutation,
each permutation of length $n$ appears precisely once in the product
$(1)^{n}$. Hence $\pi(\sigma)\equiv\frac{1}{n!}$.

Recall that the point of discussing this chain is that it is a weak
lift of the down-up chain on tableaux of the previous section, if
in the initial distribution any two permutations having the same RSK
insertion tableau are equally probable (such permutations are said
to belong to the same \emph{plactic class}). Let $\RSK$ denote the
map sending a permutation to its insertion tableau under the Robinson-Schensted-Knuth
algorithm (this tableau is often called $P$). \cite[Th. 4.3.iii]{fsym}
shows that $\fsym$ is a subalgebra of $\fqsym$ under the injection
$\theta^{*}(\mathbf{S}_{T}):=\sum_{\RSK(\sigma)=T}\mathbf{F}_{\sigma}$,
so Theorem \ref{thm:weaklumping-linearmap} applies. (Note that the
rescaling function $\eta(T)=\dim(\sh T)$ is indeed the restriction
of $\eta(\sigma)\equiv1$, since the number of terms $\mathbf{F}_{\sigma}$
in the image of $\mathbf{S}_{T}$ is $\dim(\sh T)$.) Thus
\begin{thm}
\label{thm:permutationtotableaux} The down-up Markov chain on permutations
lumps weakly to the down-up Markov chain on tableaux via taking RSK
insertion tableau, whenever the initial distribution is constant on
plactic classes. \qed
\end{thm}
This lumping can be ``concatentated'' with the lumping of Theorem
\ref{thm:tableauxtopartition} from tableaux to partitions. Thus the
down-up chain on permutations lumps weakly to the down-up chain on
partitions via taking the shape of the RSK insertion tableau, whenever
the initial distribution is constant on plactic classes. By the same
reasoning, this is true for chains from any descent operator $m\circ\Delta_{P}$.
We call such chains on permutations the $P$-shuffles-with-standardisation.
\begin{defn}
\label{def:shufflewithstandardisation} Fix an integer $n$, and let
$P(D)$ be a probability distribution on the weak-compositions of
$n$. Each step of the \emph{$P$-shuffle-with-standardisation} Markov
chain on the permutations $\sn$ (viewed in one-line notation) goes
as follows:
\begin{enumerate}
\item Choose a weak-composition $D$ of $n$ with probability $P(D)$.
\item Deconcatenate the current permutation into a word $w_{1}$ of the
first $d_{1}$ letters, $w_{2}$ of the next $d_{2}$ letters, and
so on. 
\item Replace the smallest letter in $w_{1}$ by 1, the next smallest by
2, and so on. Then replace the smallest letter in $w_{2}$ by $d_{1}+1$,
the next smallest by $d_{1}+2$, and so on for all $w_{i}$.
\item Interleave these words uniformly (i.e. uniformly choose a permutation
where the letters $1,2,\dots,d_{1}$ are in the same relative order
as in the replaced $w_{1}$, where $d_{1}+1,d_{1}+2,\dots,d_{1}+d_{2}$
are in the same relative order as in the replaced $w_{2}$, etc.).
\end{enumerate}
\end{defn}
As discussed in the previous four paragraphs and at the end of Section
II.B, the Hopf-morphisms $\fsym\hookrightarrow\fqsym$ and $\fsym\twoheadrightarrow\Lambda$
prove that 
\begin{thm}
\label{thm:cpplift} Fix an integer $n$, and let $P(D)$ be a probability
distribution on the weak-compositions of $n$. The $P$-shuffle-with-standardisation
chain on the permutations $\sn$ lumps weakly to the $P$-restriction-then-induction
chain on partitions, via taking the shape of the RSK insertion tableau,
whenever the initial distribution is constant on plactic classes.
\qed
\end{thm}
The next section deduces from this theorem a result of Fulman \cite[Th. 3.1]{jasonlift},
that the probability of obtaining a partition $\lambda$ after $t$
steps of $P$-restriction-then-induction, starting from the partition
with a single part, is the probability that the RSK shape of a deck
is $\lambda$ after $t$ iterations, starting from the identity, of
a $P$ analogue of the top-to-random shuffle.

\subsubsection{Other lumpings of $P$-shuffles-with-standardisation, strong and
weak\label{sub:fqsymlump}}

Return to the weak lumping of permutations to tableaux via RSK insertion
(i.e. ignore the second lumping to partitions). The $P$-shuffle-with-standardisation
chain has many weak lumpings in this style, thanks to the general
construction in \cite[Th. 31]{hivertcspolynomialrealisation}\cite{jbppolyrealisation}
of Hopf subalgebras of $\fqsym$ spanned by $\sum_{\theta(\sigma)=T}\mathbf{F}_{\sigma}$,
for various functions $\theta$. The criteria on $\theta:\amalg_{n}\sn\rightarrow\mathcal{C}$
(the codomain $\mathcal{C}$ can be any graded set) is that its extension
to all words with distinct letters, defined via $\theta(\sigma)=\theta(\std(\sigma))$,
should be compatible with concatenation and alphabet restriction in
the following sense:
\begin{enumerate}
\item if $\theta(\sigma_{1})=\theta(\tau_{1})$ and $\theta(\sigma_{2})=\theta(\tau_{2})$
then $\theta(\sigma_{1}\cdot\sigma_{2})=\theta(\tau_{1}\cdot\tau_{2})$;
\item if $\theta(\sigma)=\theta(\tau)$, and $\sigma_{\leftarrow r}$ (resp.
$\sigma_{r\rightarrow}$) contains the letters $1,2,\dots,r$ (resp.
$r+1,\dots,n$) in the same order as in $\sigma$, and similarly for
$\tau_{\leftarrow r}$ and $\tau_{r\rightarrow}$, then $\theta(\sigma_{\leftarrow r})=\theta(\tau_{\leftarrow r})$
and $\theta(\sigma_{r\rightarrow})=\theta(\tau_{r\rightarrow})$.
\end{enumerate}
Many such $\theta$ can be expressed in terms of insertion algorithms
similar to RSK. For example, taking $\theta$ to be binary tree insertion
\cite[Algo. 17]{hivertcspolynomialrealisation} generates the Loday-Ronco
Hopf algebra \cite{lodayroncotrees}. Thus the $P$-shuffle-with-standardisation
chain lumps weakly to a chain on binary trees, and \cite{baxter}
gives a variant with twin binary trees. Another example is the rising
sequence lengths, also known as the recoil or idescent set: $\theta(\sigma)=\Des(\sigma^{-1})$,
associated to the hypoplactic insertion of \cite[Sec. 4.8]{hypoplacticinsertion}.
Thus the $P$-shuffle-with-standardisation chain lumps weakly via
the set of rising sequence lengths.

The dual of this general construction creates quotient algebras of
the dual algebra $\fqsym^{*}$, which is isomorphic to $\fqsym$ via
inversion of permutations. These quotients satisfy the strong lumping
criterion of Theorem \ref{thm:stronglumping-hopf}, so the $P$-shuffles-with-standardisation
lump (strongly) via $\sigma\mapsto\theta(\sigma^{-1})$ for any $\theta$
satisfying the conditions above. Examples of such strong lumping maps
include RSK recording tableau, decreasing binary tree (the map $\lambda$
of \cite{lrstructure}, the ``recording'' part of the binary tree
algorithm), and descent set.

To see another example and non-example of lumpings induced from Hopf-morphisms,
consider the following commutative diagram from \cite[Th. 4.3]{fsym}:\noindent \begin{center}
\begin{tikzpicture} 
\node (A) at (0,0) {$\fqsym$, permutations};  
\node (B) at (7,0)  {$QSym$, compositions};  
\node (C) at (0,-3)  {$\fsym$, standard tableaux};   
\node (D) at (7,-3) {$\Lambda$, partitions}; 
\draw[->>] (A) -- (B) node [above, pos= 0.5] {$\Des$}; 
\draw[right hook->] (C) -- (A); 
\draw[->>] (C) -- (D) node [above, pos= 0.5] {$\sh$}; 
\draw[right hook->] (D) -- (B); 
\end{tikzpicture}
\par\end{center}The main example in Sections II.A-II.C lifts a chain on partitions
to a chain on permutations via a chain on standard tableaux, on the
bottom left. Let us see why it is not possible to construct a lift
via compositions, on the top right, instead. The corresponding Hopf
algebra here is the algebra of quasisymmetric functions \cite{qsym}.

There is no problem with the top Hopf-morphism, which sends a permutation
$\mathbf{F}_{\sigma}$ to the fundamental quasisymmetric function
$F_{\Des(\sigma)}$ associated with its descent composition. Since
this map sends a basis of $\fqsym$ to a basis of $QSym$, Theorem
\ref{thm:stronglumping-hopf} applies and $P$-shuffles-with-standardisation
lump strongly via descent composition.

The problem is with the Hopf-morphism on the right - this inclusion
is not induced from a set map from compositions to partitions. Theorem
\ref{thm:weaklumping-hopf} only applies if $s_{\lambda}=\sum_{\theta(I)=\lambda}F_{I}$
for some function $\theta$ sending compositions to partitions. This
condition does not hold, as the same $F_{I}$ can occur in the expansion
of multiple Schur functions: $s_{(3,1)}=F_{(1,3)}+F_{(2,2)}+F_{(3,1)}$,
$s_{(2,2)}=F_{(1,2,1)}+F_{(2,2)}$.

\subsection{II.D: Equidistribution of Shuffles from the Identity, with and without
standardisation\label{sec:startatid}}

In this section, we relate the $P$-shuffles-with-standardisation
of Definition \ref{def:shufflewithstandardisation}, to the ``cut-and-interleave''
shuffles of \cite{cppriffleshuffle}, which we call here \emph{$P$-shuffles-without-standardisation}
for clarity.
\begin{defn}
Fix an integer $n$, and let $P(D)$ be a probability distribution
on the weak-compositions of $n$. Each step of the \emph{$P$-shuffle-without-standardisation}
Markov chain on the permutations $\sn$ (viewed in one-line notation)
goes as follows:
\begin{enumerate}
\item Choose a weak-composition $D$ of $n$ with probability $P(D)$.
\item Deconcatenate the current permutation into a word $w_{1}$ of the
first $d_{1}$ letters, $w_{2}$ of the next $d_{2}$ letters, and
so on. 
\item Interleave these words uniformly, i.e. uniformly choose a permutation
where the letters of each $w_{i}$ stay in the same relative order.
\end{enumerate}
\end{defn}
This chain comes from the descent operators $m\circ\Delta_{P}$ applied
to the shuffle algebra of Ree \cite{cpmcfpsac}.

In \cite{jasonlift}, Fulman showed
\begin{thm}
\label{thm:rskshuffle} \cite[Th. 3.1]{jasonlift} The probability
of obtaining any partition as the RSK shape after $P$-shuffling-without-standardisation
$t$ times, starting from the identity permutation, is equal to its
probability under $t$ steps of $P$-restriction-then-induction, started
at the trivial representation (i.e. the partition with a single row).
\end{thm}
Note that the result is only about the probabilities after all $t$
steps of the chains; nothing can be deduced about the probabilities
at intermediate times. In particular, it does not assert that $P$-shuffles-without-standardisation
lumps, strongly or weakly, to $P$-restriction-then-induction; this
is false, see the discussion in the second paragraph of the introduction
to Section 3/Part II. 

Fulman remarked that this connection is ``surprising'' and ``quite
mysterious'', and perhaps a more enlightening proof is to combine
Theorem \ref{thm:cpplift} with the following.
\begin{thm}
\label{thm:multistepprobofshuffles} The distribution on permutations
after $t$ iterates of $P$-shuffles-with-standardisation is the same
as that after $t$ iterates of $P$-shuffles-without-standardisation,
if both are started from the identity permutation.
\end{thm}
The power of this theorem goes much beyond reproving Fulman's ``almost
lift'': it allows results about either type of $P$-shuffle to apply
to the other type. For example, \cite[Ex. 5.8]{hopfpowerchains} showed
that the expected number of descents after $t$ riffle-shuffles of
$n$ cards, starting at the identity, is $\left(1-2^{-t}\right)\frac{n-1}{2}$,
so this must also be the expected number of descents after $t$ riffle-shuffles-with-standardisation
starting at the identity. In the other direction, \cite[Sec. 6]{descentoperatorchains}
made the simple observation that, if one tracks only the relative
orders of the bottom $k$ cards under top-to-random-with-standardisation,
then one sees a lazy version of top-to-random-with-standardisation
on $k$ cards, lazy meaning that at each time step no move is made
with probability $\frac{n-k}{n}$. (This is a strong lumping, but
not from Hopf algebras.) Thus the distribution of the relative orders
of the bottom $k$ cards after $t$ iterates of top-to-random-without-standardisation
from the identity is the distribution after $t$ iterates of a lazy
version of top-to-random-without-standardisation on $k$ cards.

Another use of Theorem \ref{thm:multistepprobofshuffles} is to obtain
many analogues of Theorem \ref{thm:rskshuffle}, with various statistics
in place of the RSK shape. This is because $P$-shuffles-with-standardisation
is associated to $\fqsym$, which has many subquotients as noted in
Section \ref{sub:fqsymlump}/the end of Section II.C, and hence has
many strong and weak lumpings. Thus the probability distribution of
the binary search tree, the decreasing tree, the rising sequence lengths,
the descent set, and other statistics after $t$ iterates of $P$-shuffles-without-standardisation
can all be calculated from $m\circ\Delta_{P}$ Markov chains on the
statistics themselves. (As Section \ref{sub:descentsetlump}/II.E
will explain, the descent set is actually a Markov statistic of $P$-shuffling-without-standardisation.)
\begin{proof}[Proof of Theorem \ref{thm:multistepprobofshuffles}]
 The case of $t=1$ is clear, since performing a single $P$-shuffle-with-standardisation,
starting from the identity permutation, does not actually require
any standardisation. The key to showing this result for larger $t$
is to express $t$ iterates of a $P$-shuffle, with or without standardisation,
as the same $P''$-shuffle, performed only once. This uses the (vector
space) isomorphism identifying the descent operator $m\circ\Delta_{D}$
with the homogeneous noncommutative symmetric function $S^{D}$ \cite[Sec. 2.5]{descentoperatorchains}.
(See \cite{ncsym} for background on noncommutative symmetric functions.)
Write $S^{P}$ for the noncommutative symmetric function associated
to $m\circ\Delta_{P}$.

On a commutative Hopf algebra, such as the shuffle algebra, the composition
$\left(m\circ\Delta_{P}\right)\circ\left(m\circ\Delta_{P'}\right)$
corresponds to the internal product of noncommutative symmetric functions
$S^{P'}S^{P}$ \cite[Th. II.7]{descentoperators}. So $t$ iterates
of a $P$-shuffle-without-standardisation correspond to $\left(S^{P}\right)^{t}$.

Now consider the $P$-shuffles-with-standardisation, on the algebra
$\fqsym$. Note that the coproduct of the identity permutation is
\begin{equation}
\Delta(1,\cdots,n)=\sum_{r=0}^{n}(1,\cdots,r)\otimes(1,\dots,n-r),\label{eq:fqsymidcoprod}
\end{equation}
and each tensor-factor is an identity permutation of shorter length.
Thus the subalgebra of $\fqsym$ generated by identity permutations
of varying length is closed under coproduct, and $P$-shuffling-with-standardisation
from the identity stays within this sub-Hopf-algebra. Equation \ref{eq:fqsymidcoprod}
shows that this sub-Hopf-algebra is cocommutative, so a composition
of descent operators $\left(m\circ\Delta_{P}\right)\circ\left(m\circ\Delta_{P'}\right)$
corresponds to the internal product of noncommutative symmetric functions
$S^{P}S^{P'}$ \cite[Th. II.7]{descentoperators}. Despite this product
being in the opposite order from the shuffle algebra case, $t$ iterates
of $P$-shuffles-with-standardisation are also described by $\left(S^{P}\right)^{t}$.
\end{proof}
It would be interesting to find a bijective proof of this equidistribution
after $t$ steps, i.e. to find a bijection between trajectories, starting
at the identity, under $P$-shuffling with and without standardisation
that have the same endpoint. Since the products of the associated
noncommutative symmetric functions are in opposite orders for the
two chains, such a bijection should probably be ``order reversing''
in some way. For example, the trajectory below of top-to-random-without-standardisation
comes from inserting first in position 4, and then in position 2.

\noindent 
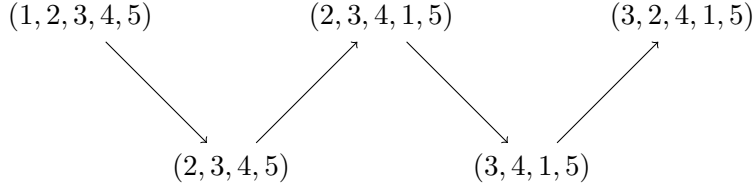
\begin{center}
\begin{tikzpicture} 
\node (A) at (0,0) {$(1,2,3,4,5)$}; 
\node (B) at (4,0)  {$(2,3,4,1,5)$}; 
\node (C) at (8,0)  {$(3,2,4,1,5)$};  
\node (E) at (2,-2) {$(2,3,4,5)$}; 
\node (F) at (6,-2)  {$(3,4,1,5)$}; 
\draw[->] (A) -- (E); 
\draw[->] (E) -- (B); 
\draw[->] (B) -- (F); 
\draw[->] (F) -- (C); 
\end{tikzpicture}
\par\end{center}

A possible trajectory of top-to-random-with-standardisation that has
the same endpoint is to insert first in position $2+1=3$ and then
in position 4:

\noindent 
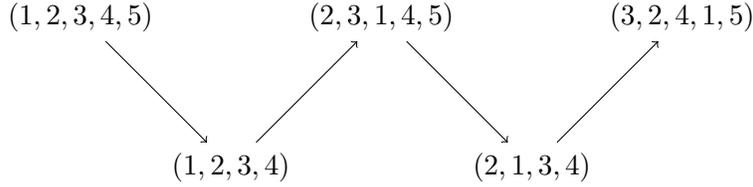
\begin{center}
\begin{tikzpicture} 
\node (A) at (0,0) {$(1,2,3,4,5)$}; 
\node (B) at (4,0)  {$(2,3,1,4,5)$}; 
\node (C) at (8,0)  {$(3,2,4,1,5)$};  
\node (E) at (2,-2) {$(1,2,3,4)$}; 
\node (F) at (6,-2)  {$(2,1,3,4)$}; 
\draw[->] (A) -- (E); 
\draw[->] (E) -- (B); 
\draw[->] (B) -- (F); 
\draw[->] (F) -- (C); 
\end{tikzpicture}
\par\end{center}

Here, the ``order reversal'' makes intuitive sense, because when
there is standardisation, 1 marks the most recently moved card, whereas
in the shuffles without standardisation, 1 is the first card moved.

To close this section, here are some similarities and differences
between the $P$-shuffles with and without standardisation. Since
both are descent operator chains, by \cite[Th. 3.5]{descentoperatorchains}
they have the same eigenvalues, but different multiplicities. (Strictly
speaking, the case without standardisation requires a multigraded
version of this theorem.) For example, the eigenvalues for both top-to-random
chains are $\frac{j}{n}$ for $j=0,1,2,\dots,n-3,n-2,n$. Their multiplicities
for the without-standardisation chain are the number of permutations
with $j$ fixed points \cite[Th. 4.1]{cppriffleshuffle}\cite[Sec. 4.6]{oneminuse},
whilst, for the chains with standardisation, they are the number of
permutations fixing $1,2,\dots,j$ but not $j+1$ (see Theorem \ref{thm:downup_evalues}
above). Hence in general the smaller eigenvalues have higher multiplicities
in the chain with standardisation. The $P$-shuffles-without-standardisation
are diagonalisable (because the shuffle algebra is commutative), and
so are the top-to-random-with-standardisation and bottom-to-random-with-standardisation
(because of a connection with dual graded graphs, see the Remark in
\cite[Sec. 4.2]{descentoperatorchains}), but Sage computer calculations
show that the $P$-shuffles-with-standardisation are generally non-diagonalisable.

\subsection{II.E: Lumping Card-Shuffles by Descent Set\label{sub:descentsetlump}}

This section addresses an example separate from the long example of
the previous sections. The goal is to apply a mild adaptation of Theorem
\ref{thm:stronglumping-hopf} to reprove the following theorem of
Athanasiadis and Diaconis:
\begin{thm}
\label{thm:shuffledescentsetlump} \cite[Ex. 5.8]{hyperplanelump}
The $P$-shuffles (without standardisation) lump via descent set.
\end{thm}
A weaker version of this result, for riffle-shuffles only, was announced
in \cite{hpmccompositions}, along with eigenvectors of the lumped
chain on compositions. The proof below is reproduced from the thesis
\cite[Sec. 6.3]{mythesis}.

Recall from Example \ref{ex:shufflealg,t2r} that the $P$-shuffles
(without standardisation) come from the descent operators $m\circ\Delta_{P}$
on the shuffle algebra $\calsh$. Thus, to prove Theorem \ref{thm:shuffledescentsetlump},
it suffices to construct a quotient Hopf algebra of $\calsh$ such
that, for each word $w\in\calsh$ with distinct letters, the quotient
map $\theta$ sends $w$ to a basis element of the quotient algebra
indexed by $\Des(w)$. This quotient Hopf algebra is $QSym$, the
quasisymmetric functions of \cite{qsym}. (Although $QSym$ is a well-known
Hopf algebra, the quotient map $\theta$ is highly non-standard, in
contrast to the Hopf-morphisms of previous sections.) The images under
$\theta$ of the words with distinct letters will be the fundamental
basis $\{F_{I}\}$, but the proof below will also require the monomial
basis $\{M_{I}\}$. Since the product and coproduct of $QSym$ are
fairly complicated, we omit the details here, and refer the interested
reader to \cite{qsym}. 
\begin{thm}
\label{thm:shuffletoqsym} \cite[Th. 6.2.1]{mythesis} There is a
morphism of Hopf algebras $\theta:\calsh\rightarrow QSym$ such that,
if $w$ is a word with distinct letters, then $\theta(w)=F_{\Des(w)}$.\end{thm}
\begin{proof}
By \cite[Th. 4.1]{qsymisterminal}, $QSym$ is the terminal object
in the category of combinatorial Hopf algebras equipped with a multiplicative
character. So, to define any Hopf-morphism to $QSym$, it suffices
to define the corresponding character $\zeta$ on the domain. By \cite[Th. 6.1.i]{freeliealgs},
the shuffle algebra $\calsh$ is freely generated by \emph{Lyndon
words} \cite[Sec. 5.1]{lothaire}, which are strictly smaller than
all their cyclic rearrangements. Hence any choice of the values of
$\zeta$ on Lyndon words extends uniquely to a well-defined character
on $\calsh$. For Lyndon $u$, set 
\begin{equation}
\zeta(u)=\begin{cases}
1 & \mbox{if }u\mbox{ has all letters distinct and has no descents};\\
0 & \mbox{otherwise.}
\end{cases}\label{eq:defzeta}
\end{equation}
We claim that, consequently, (\ref{eq:defzeta}) holds for all words
with distinct letters, even if they are not Lyndon. Assuming this
for now, \cite[Th. 4.1]{qsymisterminal} defines 
\begin{align*}
\theta(w) & =\sum_{I\vdash n}(\zeta\otimes\dots\otimes\zeta)\left(\Delta_{I}(w)\right)M_{I}\\
 & =\sum_{I\vdash n}\zeta((w_{1},\dots,w_{i_{1}}))\zeta((w_{i_{1}+1},\dots,w_{i_{1}+i_{2}}))\dots\zeta((w_{i_{l(I)-1}+1},\dots,w_{n}))M_{I}.
\end{align*}
If $w$ has distinct letters, then every consecutive subword $(w_{i_{1}+\dots+i_{j}+1},\dots,w_{i_{1}\dots+i_{j+1}})$
of $w$ also has distinct letters, so 
\[
\zeta((w_{1},\dots,w_{i_{1}}))\dots\zeta((w_{i_{l(I)-1}+1},\dots,w_{n}))=\begin{cases}
1 & \mbox{if }\Des(w)\leq I;\\
0 & \mbox{otherwise.}
\end{cases}
\]
Hence $\theta(w)=\sum_{\Des(w)\leq I}M_{I}=F_{\Des(w)}$.

Now return to proving the claim that (\ref{eq:defzeta}) holds whenever
$w$ has distinct letters. Proceed by induction on $w$, with respect
to lexicographic order. \cite[Th. 6.1.ii]{freeliealgs}, applied to
a word $w$ with distinct letters, states that: if $w$ has Lyndon
factorisation $w=u_{1}\cdot\dots\cdot u_{k}$, then the product of
these factors in the shuffle algebra satisfies 
\[
u_{1}\shuffle\dots\shuffle u_{k}=w+\sum_{v<w}\alpha_{v}v
\]
 where $\alpha_{v}$ is 0 or 1. The character $\zeta$ is multiplicative,
so 
\begin{equation}
\zeta(u_{1})\dots\zeta(u_{k})=\zeta(w)+\sum_{v<w}\alpha_{v}\zeta(v).\label{eq:zeta}
\end{equation}
If $w$ is Lyndon, then the claim is true by definition; this includes
the base case for the induction. Otherwise, $k>1$ and there are two
possibilities:
\begin{itemize}
\item None of the $u_{i}$s have descents. Then the left hand side of (\ref{eq:zeta})
is 1. Since the $u_{i}$s together have all letters distinct, the
only way to shuffle them together and obtain a word with no descents
is to arrange the constituent letters in increasing order. This word
is Lyndon, so it is not $w$, and, by inductive hypothesis, it is
the only $v$ in the sum with $\zeta(v)=1$. So $\zeta(w)$ must be
0.
\item Some Lyndon factor $u_{i}$ has at least one descent. Then $\zeta(u_{i})=0$,
so the left hand side of (\ref{eq:zeta}) is 0. Also, no shuffle of
$u_{1},\dots,u_{k}$ has its letters in increasing order. Therefore,
by inductive hypothesis, all $v$ in the sum on the right hand side
have $\zeta(v)=0$. Hence $\zeta(w)=0$ also.
\end{itemize}
\end{proof}
\begin{rem*}
From the proof, one sees that the conclusion $\theta(w)=F_{\Des(w)}$
for $w$ with distinct letters relies only on the value of $\zeta$
on Lyndon words with distinct letters. The proof took $\zeta(u)=0$
for all Lyndon $u$ with repeated letters, but any other value would
also work. Alas, no definition of $\zeta$ will ensure that the images
of all words are $F_{I}$ for some $I$:
\[
\theta((1,1))=\frac{1}{2}\theta((1)\shuffle(1))=\frac{1}{2}\theta(1)\theta(1)=\frac{1}{2}M_{(1)}^{2}=F_{(1,1)}+F_{(2)}.
\]

\end{rem*}


\newcommand{\etalchar}[1]{$^{#1}$}

\end{document}